\documentclass[11pt,oneside, reqno]{amsart}
\usepackage[top=25mm, bottom=25mm, left=25mm, right=25mm]{geometry}
\usepackage{amsfonts,amssymb,amsmath,amsthm,amsxtra}
\usepackage[T1]{fontenc}
\usepackage[utf8]{inputenc}
\usepackage{bm}
\usepackage{mathtools}
\usepackage{dsfont}
\usepackage{bbm}
\usepackage{mathrsfs}
\usepackage{enumitem}
\usepackage{xcolor}

\usepackage{graphics}

\numberwithin{equation}{section}
\newtheorem{theorem}[equation]{Theorem}
\newtheorem{corollary}[equation]{Corollary}

\newtheorem{proposition}[equation]{Proposition}
\newtheorem{conjecture}[equation]{Conjecture}

\theoremstyle{definition}
\newtheorem{example}[equation]{Example}
\newtheorem{remark}[equation]{Remark}

\renewcommand{\AA}{\mathbb{A}}
\newcommand{\BB}{\mathbb{B}}
\newcommand{\CC}{\mathbb{C}}
\newcommand{\DD}{\mathbb{D}}
\newcommand{\EE}{\mathbb{E}}

\newcommand{\GG}{\mathbb{G}}

\newcommand{\II}{\mathbb{I}}
\newcommand{\JJ}{\mathbb{J}}

\newcommand{\LL}{\mathbb{L}}

\newcommand{\NN}{\mathbb{N}}

\newcommand{\RR}{\mathbb{R}}

\newcommand{\TT}{\mathbb{T}}

\newcommand{\ZZ}{\mathbb{Z}}

\newcommand{\I}{\mathbb{I}}
\newcommand{\C}{\mathbb{C}}
\newcommand{\N}{\mathbb{N}}

\newcommand{\calC}{\mathcal{C}}
\newcommand{\calP}{\mathcal{P}}
\newcommand{\calF}{\mathcal{F}}
\newcommand{\calB}{\mathcal{B}}

\newcommand{\calT}{\mathcal{T}}
\newcommand{\calN}{\mathcal{N}}

\newcommand{\calS}{\mathcal{S}}

\newcommand{\dif}{\mathrm{d}}
\newcommand{\ex}{\bm{e}}

\newcommand{\ind}[1]{\mathds{1}_{{#1}}}

\newcommand*{\DMO}[1]{\expandafter\DeclareMathOperator\csname #1\endcsname {#1}}
\DMO{ord}
\DMO{supp}
\DMO{Sym}
\DMO{cl}
\DMO{lcm}
\DMO{dist}
\DMO{tr}
\DMO{diam}

\DeclarePairedDelimiter\abs{\lvert}{\rvert}

\DeclarePairedDelimiter\norm{\lVert}{\rVert}

\DeclarePairedDelimiterX\spr[2]{\langle}{\rangle}{#1,#2}
\newcommand{\ipr}[2]{#1\cdot#2}
\DeclarePairedDelimiterX\Set[2]{\{}{\}}{#1\colon #2}
\DeclarePairedDelimiterX\Seq[1]{(}{)}{#1}

\def\8{\infty}

\newcommand{\vertiii}[1]{{\left\lvert\kern-0.25ex\left\lvert\kern-0.25ex\left\lvert #1 
    \right\rvert\kern-0.25ex\right\rvert\kern-0.25ex\right\rvert}} 

\begin{document}

\title[Oscillation inequalities in ergodic theory and analysis]{Oscillation inequalities in ergodic theory and analysis: one-parameter and multi-parameter perspectives}

\author{Mariusz Mirek }
\address[Mariusz Mirek]{
Department of Mathematics,
Rutgers University,
Piscataway, NJ 08854-8019, USA 
\&  School of Mathematics,
  Institute for Advanced Study,
  Princeton, NJ 08540,
  USA
\&
Instytut Matematyczny,
Uniwersytet Wroc{\l}awski,
Plac Grunwaldzki 2/4,
50-384 Wroc{\l}aw,
Poland}
\email{mariusz.mirek@rutgers.edu}

\author{Tomasz Z. Szarek} 
\address[Tomasz Z. Szarek]{
	BCAM - Basque Center for Applied Mathematics,
	48009 Bilbao, Spain  
	\&
Instytut Matematyczny,
Uniwersytet Wroc{\l}awski,
Plac Grunwaldzki 2/4,
50-384 Wroc{\l}aw,
Poland}
\email{tzszarek@bcamath.org}

\author{James Wright}
\address[James Wright]{Maxwell Institute of Mathematical Sciences and The School of Mathematics,
The University of Edinburgh
James Clerk Maxwell Building,
The King's Buildings,
Peter Guthrie Tait Road,
City Edinburgh,
EH9 3FD}
\email{J.R.Wright@ed.ac.uk}

\thanks{Mariusz Mirek was partially supported by NSF grant DMS-2154712, and by the National Science Centre
in Poland, grant Opus 2018/31/B/ST1/00204.  Tomasz Z. Szarek was
partially supported by the National Science Centre of Poland, grant
Opus 2017/27/B/ST1/01623, by Juan de la Cierva Incorporaci{\'o}n 2019
grant number IJC2019-039661-I funded by Agencia Estatal de
Investigaci{\'o}n, grant
PID2020-113156GB-I00/AEI/10.13039/501100011033 and also by the Basque
Government through the BERC 2018-2021 program and by Spanish Ministry
of Sciences, Innovation and Universities: BCAM Severo Ochoa
accreditation SEV-2017-0718.}

\begin{abstract}
In this survey we review useful tools that naturally arise in
the study of pointwise convergence problems in analysis, ergodic
theory and probability.  We will pay special attention to
quantitative aspects of pointwise convergence phenomena from the point
of view of oscillation estimates in both the single and several parameter 
settings. We establish a number of new oscillation inequalities and give new proofs for
known results with elementary arguments.
\end{abstract}


\maketitle

\section{Introduction}\label{section:1}

Pointwise convergence is the most natural as well as the most
difficult type of convergence to establish. It requires 
sophisticated tools in analysis, ergodic theory and probability. In this survey,
we will review variation and oscillation semi-norms as well as the $\lambda$-jump counting function
which give us quantitative measures for pointwise convergence. However we will concentrate
on the central role that oscillation inequalities play, both in the one-parameter and multi-parameter
settings.

In the one-parameter setting we derive a simple abstract
oscillation estimate for the so-called projective operators, which will result in oscillation estimates for
martingales, smooth bump functions as well as the Carleson operator.  The
multi-parameter oscillation semi-norm is the only available tool that
allows us to handle efficiently multi-parameter pointwise convergence
problems with arithmetic features.  This contrasts sharply with the
one-parameter setting, where we have a variety of tools including
oscillations, variations or $\lambda$-jumps to handle pointwise convergence
problems.  The multi-parameter oscillation estimates will be
illustrated in the context of the Dunford--Zygmund ergodic theorem for
commuting measure-preserving transformations as well as 
observations of Bourgain for certain multi-parameter polynomial ergodic
averages.

We begin with describing methods that permit us to
handle pointwise convergence problems in the context of various
ergodic averaging operators. Before we do this we set up notation and
terminology, which will allow us to discuss various concepts in a fairly unified way. 

Throughout this survey the triple $(X, \mathcal B(X), \mu)$ denotes a
$\sigma$-finite measure space. The space of all formal $k$-variate polynomials
$P({\rm m}_1, \ldots, {\rm m}_k)$ with $k\in\ZZ_+$ indeterminates
${\rm m}_1, \ldots, {\rm m}_k$ and integer coefficients will be denoted by $\ZZ[{\rm m}_1, \ldots, {\rm m}_k]$. We will always identify each polynomial $P\in\ZZ[{\rm m}_1, \ldots, {\rm m}_k]$ with a function
$(m_1,\ldots, m_k)\mapsto P(m_1,\ldots, m_k)$ from $\ZZ^k$ to $\ZZ$.

Let $d, k \in\ZZ_+$, and consider a family
${\mathcal T} = (T_1,\ldots, T_d)$ of invertible commuting
measure-preserving transformations on $X$, polynomials
${\mathcal P} = (P_1,\ldots, P_d) \subset \ZZ[\mathrm m_1, \ldots, \rm m_k]$,
an integer $k$-tuple $M = (M_1,\ldots, M_k)\in\ZZ_+^k$, and a measurable
function $f:X\to\CC$. We consider the multi-parameter polynomial ergodic
average 
\begin{equation*}
\frac{1}{M_1 \cdots M_k} \sum_{m_1=1}^{M_1} \cdots \sum_{m_k=1}^{M_k}
f\bigl(T_1^{P_1({ m_1},\ldots,{ m_k})}
 \cdots T_d^{P_d({ m_1},\ldots,{ m_k})} x\bigr).
\end{equation*}
We denote this average by $A_{{M};X,{\mathcal T}}^{\mathcal P}f(x)$ and we use the notation
\begin{align}
\label{eq:31}
A_{{M}; X, {\mathcal T}}^{\mathcal P}f(x): =  \EE_{m\in Q_{M}}f(T_1^{P_1(m)}\cdots T_d^{P_d(m)}x), \qquad x\in X,
\end{align}
where $Q_{M}:=[M_1]\times\ldots\times[M_k]$ is a box in $\ZZ^k$ with $[N]:=(0, N]\cap\ZZ$ for any real number $N\ge1$, and $\EE_{y\in Y}f(y):=\frac{1}{\#Y}\sum_{y\in Y}f(y)$ for any finite set $Y$ and any  $f:Y\to\CC$.
We will often abbreviate $A_{M; X, {\mathcal T}}^{{\mathcal P}}$ to  $A_{M; X}^{{\mathcal P}}$ when the tranformations are understood.
Depending on how explicit we want to be, more precision may be necessary and 
we will write out the averages
\begin{align*}
A_{M;X}^{\mathcal P}f(x) =A_{M_1,\ldots, M_k;X}^{P_1,\ldots, P_d} f(x)
\qquad \text{ or } \qquad
A_{{M}; X, {\mathcal T}}^{\mathcal P}f(x) =
A_{M_1,\ldots, M_k;X,T_1,\ldots, T_d}^{P_1,\ldots, P_d} f(x).
\end{align*}

\begin{example}\label{ex:1}
Due to the Calder{\'o}n transference principle \cite{Cald}, the most
important dynamical system, from the point of view of pointwise
convergence problems, is the integer shift system.
Namely, it is the
$d$-dimensional lattice $(\ZZ^d, \mathcal B(\ZZ^d), \mu_{\ZZ^d})$
equipped with a family of shifts $S_1,\ldots, S_d:\ZZ^d\to\ZZ^d$,
where $\mathcal B(\ZZ^d)$ denotes the $\sigma$-algebra of all subsets
of $\ZZ^d$, $\mu_{\ZZ^d}$ denotes counting measure on $\ZZ^d$, and
$S_j(x):=x-e_j$ for every $x\in\ZZ^d$ (here $e_j$ is $j$-th basis
vector from the standard basis in $\ZZ^d$ for each $j\in[d]$). Then the
average $A_{M; X, {\mathcal T}}^{{\mathcal P}}$ from \eqref{eq:31} with
${\mathcal T} = (T_1,\ldots, T_d)=(S_1,\ldots, S_d)$ can be rewritten
for any $x=(x_1,\ldots, x_d)\in\ZZ^d$ and any finitely supported
function $f:\ZZ^d\to\CC$ as
\begin{align}
\label{eq:32}
A_{M; \ZZ^d}^{{\mathcal P}}f(x)=\EE_{m\in Q_{M}}f(x_1-P_1(m),\ldots, x_d-P_d(m)).
\end{align}
\end{example}

\subsection{Birkhoff's and von Neumann's ergodic theorems} In the
early 1930's Birkhoff \cite{BI} and von Neumann \cite{vN} established an
almost everywhere pointwise ergodic theorem and a mean ergodic theorem,
respectively, which we summarize in the following result.
\begin{theorem}[Birkhoff's and von Neumann's ergodic theorem]
\label{birk}
Let $(X,\mathcal B(X), \mu)$ be a $\sigma$-finite measure space equipped
with a measure-preserving transformation $T:X\to X$. Then for every
$p\in[1, \infty)$ and every $f\in L^p(X)$ the averages
\[
A_{M;X, T}^{\mathrm m}f(x)=\EE_{m\in [M]}f(T^{m}x), \qquad x\in X, \qquad M\in\ZZ_+
\]
converge almost everywhere on $X$ and in $L^p(X)$ norm as
$M\to\infty$.
\end{theorem}

Although there are many proofs of Theorem \ref{birk} in the
literature, (we refer for instance to the monographs \cite{EW, Peter}
for more details and the historical background), there is a particular proof
which is important in our context. This proof illustrates the classical strategy for
handling pointwise convergence problems, which is based on a two-step
procedure:
\begin{itemize}
\item[(i)] The first step establishes $L^p(X)$ boundedness (when
$p\in(1, \infty)$), or a weak type $(1, 1)$ bound (when $p=1$) of the
corresponding maximal function
$\sup_{M\in\ZZ_+}|A_{M;X, T}^{\mathrm m}f(x)|$.  This in turn, using
the Calder{\'o}n transference principle \cite{Cald}, can be derived
from the corresponding maximal bounds for the Hardy--Littlewood
maximal function $\sup_{M\in\ZZ_+}|A_{M;\ZZ}^{\mathrm m}f(x)|$ on the
set of integers, see \eqref{eq:32}. Having these maximal estimates in hand one can 
easily prove that the set
\begin{align*}
\qquad \qquad \mathbf P\mathbf C[L^{p}(X)]=\{f\in L^{p}(X): \lim_{M\to\infty}A_{M;X, T}^{\mathrm m}f \text{
exists $\mu$-almost everywhere on $X$}\}
\end{align*}
is closed in $L^{p}(X)$. 
\item[(ii)] In the second step one shows that
$\mathbf P\mathbf C[L^{p}(X)]= L^{p}(X)$. In view of the first step
the task is reduced to finding a dense class of functions in
$L^{p}(X)$ for which we have pointwise convergence. In our problem let us first assume $p=2$.
Then invoking a variant of Riesz decomposition \cite{Riesz} a good
candidate is the space ${\rm I}_T\oplus {\rm J}_T\subseteq L^2(X)$, where
\begin{align*}
\qquad\qquad {\rm I}_T:=\{f\in L^2(X): f\circ T=f\},
\qquad \text{ and } \qquad
{\rm J}_T:=\{g-g\circ T: g\in L^2(X)\cap L^{\infty}(X)\}.
\end{align*}
We then note that $A_{M;\ZZ}^{\mathrm m}f=f$ for  $f\in {\rm I}_T$, and $\lim_{M\to\infty}A_{M;X, T}^{\mathrm m}h=0$ for $h\in {\rm J}_T$, since
\[
A_{M;X, T}^{\mathrm m}h=M^{-1}\big(g\circ T-g\circ T^{M+1}\big)
\]
telescopes, whenever $h=g-g\circ T\in {\rm J}_T$. This establishes pointwise almost
everywhere convergence of $A_{M;X, T}^{\mathrm m}$ on ${\rm I}_T\oplus {\rm J}_T$, which is dense
in $L^2(X)$. These two steps guarantee that $\mathbf P\mathbf C[L^{2}(X)]= L^{2}(X)$. Consequently,
$A_{M;X, T}^{\mathrm m}$ converges pointwise on $L^{p}(X)\cap L^{2}(X)$ for any $p\in[1, \infty)$. Since 
$L^{p}(X)\cap L^{2}(X)$ is dense in $L^{p}(X)$  we also conclude, in view of the first step, that
$\mathbf P\mathbf C[L^{p}(X)]= L^{p}(X)$, and this completes a brief outline of the proof of Theorem~\ref{birk}.
\end{itemize}

\subsection{Dunford--Zygmund pointwise ergodic theorem}
In the early 1950's it was observed by Dunford \cite{D} and
independently by Zygmund \cite{Z} that the two-step procedure
 can be applied in a multi-parameter setting. More
precisely, the Dunford--Zygmund multi-parameter pointwise ergodic theorem, where the convergence is understood in the unrestricted sense, 
 can be formulated as follows.

\begin{theorem}[Dunford--Zygmund ergodic theorem]
\label{dz}
Let $d\in\ZZ_+$ and let $(X,\mathcal B(X), \mu)$ be a $\sigma$-finite measure space equipped with
a family $\mathcal T=(T_1,\ldots, T_d)$ of not necessarily commuting and measure-preserving transformations $T_1,  \ldots, T_d:X\to X$.
Then for every
$p\in(1, \infty)$ and every $f\in L^p(X)$ the averages
\[
A_{M_1,\ldots, M_d;X, \mathcal T}^{{\mathrm m}_1,\ldots, {\rm m}_d}f(x)
=\EE_{m\in Q_{M}}f(T_1^{m_1}\cdots T_d^{m_d}x), \qquad x\in X, \qquad M = (M_1,\ldots, M_d)\in\ZZ_+^d,
\]
converge almost everywhere on $X$ and in $L^p(X)$ norm as
$\min\{M_1,\ldots, M_d\}\to\infty$.
\end{theorem}

\vskip 5pt
This theorem has a fairly simple proof, which is based on the following identity
\begin{align*}
A_{M_1,\ldots, M_d;X, \mathcal T}^{{\mathrm m}_1,\ldots, {\rm m}_d}f=
A_{M_1;X, T_1}^{{\mathrm m}_1}\circ \ldots \circ A_{M_d;X, T_d}^{{\mathrm m}_d}f.
\end{align*}
The $L^p(X)$ bounds (with $p\in(1, \infty]$) for the strong
maximal function
$\sup_{M\in\ZZ_+^d}|A_{M_1,\ldots, M_d;X, \mathcal T}^{{\mathrm m}_1,\ldots, {\rm m}_d}f|$
follow easily by applying $d$ times the corresponding $L^p(X)$ bounds
for $\sup_{M\in\ZZ_+}|A_{M;X, T}^{\mathrm m}f|$. This establishes the
first step in the two-step procedure described above. The second step
is based on a suitable adaptation of the telescoping argument to the
multi-parameter setting and an application of the classical Birkhoff
ergodic theorem, see \cite{Nevo} for more details.  These two steps establish Theorem~\ref{dz} and
motivates our further discussion on multi-parameter
convergence problems.  One also knows that pointwise convergence in
Theorem \ref{dz} may fail if $p = 1$, and that the operator
$f \mapsto \sup_{M\in\ZZ_+^d}|A_{M_1,\ldots, M_d;X, \mathcal T}^{{\mathrm m}_1,\ldots, {\rm m}_d}f|$
is not of weak type $(1, 1)$ in general (even if we assume that the transformations $T_j$, $1 \le j \le d$, commute). A model example is $X=\mathbb{Z}^d$ and $T_j x = x - e_j$, $1 \le j \le d$, where $e_j$ is the $j$th coordinate vector. Then the corresponding maximal operator is just the strong maximal operator for which it is well known that the weak type $(1,1)$ estimate does not hold.

\subsection{Quantitative tools in the study of pointwise convergence}
The approach described in the context of Theorem \ref{birk} and
Theorem \ref{dz} has a quantitative nature but it says nothing
quantitatively about pointwise convergence. This approach is very
effective in pointwise convergence questions arising in harmonic
analysis as there are many natural dense subspaces in
Euclidean settings, which can be used to establish pointwise
convergence. However for ergodic theoretic questions, when one works
with abstract measure spaces, the situation is dramatically
different as Bourgain showed \cite{B1,B2,B3}. We shall see more
examples below.

Consequently, the second step from the two-step procedure may require
more quantitative tools to establish pointwise convergence. To
overcome the difficulties with determining dense subspace for which
pointwise convergence may be verified, Bourgain \cite{B3} proposed
three other approaches.

\begin{enumerate}[label*={\arabic*}.]

\item The first approach is based on controlling the so-called oscillation
semi-norms. Let $\JJ\subseteq \NN$ be so that $\#\JJ\ge2$,
let $I=(I_j:j\in\NN_{\le J})$ be a strictly increasing sequence of length $J+1$ for some
$J\in\ZZ_+$, which takes values in $\JJ$, and recall that for any
sequence $(\mathfrak a_{t}: t\in\JJ)\subseteq\CC$,  and any exponent
$1 \leq r < \infty$, the $r$-oscillation
seminorm is defined by
\begin{align}
\label{eq:38}
O_{I, J}^r(\mathfrak a_{t}: t \in \JJ):=
\Big(\sum_{j=0}^{J-1}\sup_{\substack{I_j\le t<I_{j+1}\\ t\in \JJ}}
\abs{\mathfrak a_{t} - \mathfrak a_{I_j}}^r\Big)^{1/r}.
\end{align}
We will give a more general definition of $r$-oscillations in the multi-parameter
setting; see \eqref{eq:102}.

\item The second approach is based on controlling the so-called
$r$-variation seminorms. For any $\mathbb I\subseteq \NN$, any sequence
$(\mathfrak a_t: t\in\mathbb I)\subseteq \CC$, and any exponent
$1 \leq r < \infty$, the $r$-variation semi-norm is defined to be
\begin{align*}
V^{r}(\mathfrak a_t: t\in\mathbb I):=
\sup_{J\in\ZZ_+} \sup_{\substack{t_{0}<\dotsb<t_{J}\\ t_{j}\in\mathbb I}}
\Big(\sum_{j=0}^{J-1}  |\mathfrak a_{t_{j+1}}-\mathfrak a_{t_{j}}|^{r} \Big)^{1/r},
\end{align*}
where the latter supremum is taken over all finite increasing sequences in
$\mathbb I$. 

\item The third approach is based on studying the $\lambda$-jump
counting function which is closely related to 
$r$-variations. For any $\II\subseteq \NN$ and any $\lambda>0$, the $\lambda$-jump
counting function of a sequence  $(\mathfrak a_t: t\in\mathbb I)\subseteq \CC$ is defined by
\begin{align*}
N_\lambda(\mathfrak a_t:t\in\II):=\sup \{ J\in\N : \exists_{\substack{t_{0}<\ldots<t_{J}\\ t_{j}\in\I}}
: \min_{0 \le j \le J-1} |\mathfrak a_{t_{j+1}}-\mathfrak a_{t_{j}}| \ge \lambda \}.
\end{align*}
We also refer to Section \ref{section:2} for simple properties of
$r$-oscillations, $r$-variations and $\lambda$-jumps. These will be illustrated
in the context of bounded martingales, a toy model
explaining their quantitative nature 
and their usefulness in pointwise
convergence problems.

\end{enumerate}

\subsection{Bourgain's pointwise ergodic theorem}
In the early 1980's Bellow \cite{Bel} (being motivated by some
problems from equidistribution theory) and independently Furstenberg
\cite{F} (being motivated by some problems from additive combinatorics
in the spirit of Szemer{\'e}di's theorem \cite{Sem1} for arithmetic
progressions) posed the problem of whether for any polynomial
$P\in \ZZ[{\rm m}]$ and any measure-preserving map $T:X\to X$ on a
probability space $(X,\mathcal B(X), \mu)$, the averages
\begin{align}
\label{eq:8}
A_{M;X, T}^{P(\mathrm m)}f(x)=\EE_{m\in [M]}f(T^{P(m)}x), \qquad x\in X, \qquad M\in\ZZ_+
\end{align}
converge almost everywhere on $X$ as $M\to\infty$, for any
$f\in L^{\infty}(X)$.

An affirmative answer to this question was given by Bourgain in series
of groundbreaking papers \cite{B1,B2,B3} which we summarize in
the following theorem.
\begin{theorem}[Bourgain's ergodic theorem]
\label{bourgain}
Let $(X,\mathcal B(X), \mu)$ be a $\sigma$-finite measure space
equipped with an invertible measure-preserving transformation
$T:X\to X$. Assume that $P\in \ZZ[{\rm m}]$ is a polynomial such that
$P(0)=0$. Then for every $p\in(1, \infty)$ and every $f\in L^p(X)$ the
averages $A_{M;X, T}^{P}f$ from \eqref{eq:8}
converge almost everywhere on $X$ and in $L^p(X)$ norm as
$M\to\infty$.
\end{theorem}

Theorem \ref{bourgain} is an instance where establishing
pointwise convergence on a dense class is a challenging problem.  The
decomposition ${\rm I}_T\oplus {\rm J}_T$ of von Neumann (as for
$A_{M;X, T}^{{\mathrm m}}$) is not sufficient if $\deg P\ge2$, though it
still makes sense.  Even for the squares $P(m)=m^2$ it is not clear
whether $\lim_{M\to\infty}A_{M;X, T}^{{\mathrm m}^2}h=0$ for
$h\in {\rm J}_T$. The reason is that the averages
$A_{M;X, T}^{{\mathrm m}^2}h$ do not telescope for $h\in {\rm J}_T$
anymore, since the differences $(m+1)^2-m^2=2m+1$ have unbounded gaps.

Nearly two decades after Bourgain papers \cite{B1,B2,B3}, it was discovered
that the range of $p\in(1, \infty)$ in Bourgain's theorem is
sharp. In contrast to Birkhoff's theorem, if $P\in\ZZ[\rm m]$ is a
polynomial of degree at least two, the pointwise convergence at the
endpoint for $p=1$ may fail as was shown by Buczolich and Mauldin
\cite{BM} for $P(m)=m^2$ and by LaVictoire \cite{LaV1} for $P(m)=m^k$
for any $k\ge2$. This also stands in sharp contrast to what happens for 
continuous analogues of ergodic averages
and shows that any intuition that we build
in Euclidean harmonic analysis (when sums are replaced with integrals) can fail dramatically in 
discrete problems.

Bourgain \cite{B1,B2,B3} also used the two-step procedure to prove
Theorem \ref{bourgain}. In the first step, it was proved that for all
$p\in(1, \infty]$, there exists $C_{p, P}>0$ such that for every
$f\in L^p(X)$ we have
\begin{align}
\label{eq:33}
\norm[\big]{\sup_{M\in\ZZ_+}\abs{A_{M; X, T}^{P}f}}_{L^p(X)}
\le C_{p, P}\norm{f}_{L^p(X)}.
\end{align}

However, in the second step of the two-step procedure a quantitative
pointwise ergodic theorem was established by studying oscillation
semi-norms, see \eqref{eq:38}.  More, precisely, it was proved that
for any $\tau>1$, any sequence of integers
$I=(I_j:{j\in\NN})\subseteq\LL_{\tau}:=\{\lfloor \tau^n\rfloor: n\in\NN\}$
such that $I_{j+1}>2I_j$ for all $j\in\NN$, and any $f\in L^2(X)$ one
has
\begin{align}
\label{eq:34}
\norm[\big]{O_{I, J}^2(A_{M; X, T}^{P}f:M\in\LL_\tau)}_{L^2(X)}
\le C_{I, \tau}(J)\norm{f}_{L^2(X)}, \qquad J\in\ZZ_+,
\end{align}
where $C_{I, \tau}(J)$ is a constant depending on $I$ and $\tau$ and
satisfies
\begin{align}
\label{eq:41}
\lim_{J\to\infty} J^{-1/2}C_{I, \tau}(J)=0.
\end{align}

Bourgain \cite{B1, B2, B3} had the ingenious insight to see that
inequality \eqref{eq:34} with \eqref{eq:41} suffices to establish
pointwise convergence of $A_{M; X, T}^{P}f$ for any
$f\in L^2(X)$. Inequality \eqref{eq:34} with \eqref{eq:41} can be
thought of as the weakest possible quantitative form for pointwise
convergence. On the one hand, \eqref{eq:34} is very close to the
maximal inequality, since by using \eqref{eq:33} with $p=2$ we can derive \eqref{eq:34} with a 
constant at most $J^{1/2}$. On the other hand, any
improvement (better than $J^{1/2}$)  for the
constant in \eqref{eq:34} implies \eqref{eq:41} and so ensures
pointwise convergence of $A_{M; X, T}^{P}f$ for any
$f\in L^2(X)$, see Proposition \ref{prop:4}, where the details, even
in the multi-parameter setting, are given.  Therefore, from this point
of view, inequality \eqref{eq:34} with \eqref{eq:41} is the minimal
quantitative requirement necessary to establish pointwise convergence.

Bourgain's papers \cite{B1, B2, B3} were a significant breakthrough in
ergodic theory, which used a variety of new tools (ranging from
harmonic analysis and number theory through probability and the theory of
Banach spaces) to study pointwise convergence problems in analysis
understood in a broad sense. In \cite{B3}
a complete proof of Theorem \ref{bourgain} is given using the notions of
$r$-variations and $\lambda$-jumps (introduced by Pisier and Xu \cite{PX}), which are
two important quantitative tools in the study of pointwise
convergence problems. This initiated a systematic study of
quantitative estimates in harmonic analysis and ergodic theory which
resulted in a vast literature: in ergodic theory \cite{BM, jkrw, JR1,
JRW, JRW1, LaV1, MST2, Kr, KMT, MSS, MSZ3}, in discrete harmonic analysis
\cite{IMMS, IMSW, IMW, IW, M10, MSZ1, MSW, MT, Pierce, RW}, and in
classical harmonic analysis \cite{BORSS, BRS, CJRW, DDU, GRY, JG, JSW, La1, La2, OSTTW, MST1, MSZ2}. 

Not long after \cite{B3}, Lacey refined Bourgain's argument
\cite[Theorem 4.23, p. 95]{RW}, and showed that for every $\tau>1$
there is a constant $C_{\tau}>0$ such that for any $f\in L^2(X)$ one
has
\begin{align}
\label{eq:35}
\sup_{J\in\ZZ_+}\sup_{I\in \mathfrak S_J(\LL_{\tau})}
\norm[\big]{O_{I, J}^2(A_{M; X, T}^{P}f:M\in\LL_\tau)}_{L^2(X)}
\le C_{\tau}\norm{f}_{L^2(X)},
\end{align}
where $\mathfrak S_J(\LL_{\tau})$ denotes the set of all strictly
increasing sequences $I=(I_j:j\in\NN_{\le J})\subset \LL_{\tau}$ of
length $J+1$ for some $J\in\ZZ_+$.  Inequality \eqref{eq:35} was the
first uniform oscillation result in the class of $\tau$-lacunary
sequences.  Lacey's observation naturally motivated a question (which
also motivates this survey) whether there are uniform estimates, independent of
$\tau>1$, of oscillation inequalities in \eqref{eq:35}. For the
Birkhoff averages $A_{M;X, T}^{{\mathrm m}}$, this was explicitly formulated
in \cite[Problem 4.12, p. 80]{RW}. We will discuss below uniform oscillation
estimates as well as other quantitative forms of pointwise convergence
including $r$-variations and $\lambda$-jumps.

\subsection{Martingales: a model to study pointwise convergence problems}
In order to understand the relationship between $r$-oscillations,
$r$-variations and $\lambda$-jumps, we will use bounded martingales
$\mathfrak f=(\mathfrak f_n:X\to \CC:n\in\ZZ_+)$ as a toy model to
help us understand the connections and various nuances. All properties that will be used
in the discussion below are collected in Section \ref{section:2}. The
discussion will follow the development of the various notions in chronological order.

The $r$-variations for
$\mathfrak f=(\mathfrak f_n:X\to \CC:n\in\ZZ_+)$ were investigated by L{\'e}pingle
\cite{Lep} who established that for all $r\in(2, \infty)$ and
$p\in(1, \infty)$, there is a constant $C_{p, r}>0$ such that
\begin{align}
\label{eq:42}
\norm{V^r(\mathfrak f_n: n\in\ZZ_+)}_{L^p(X)}
\le C_{p, r}\sup_{n\in\ZZ_+}\norm{\mathfrak f_n}_{L^p(X)}.
\end{align}
In fact, L{\'e}pingle \cite{Lep} also proved a weak type $(1, 1)$
estimate.  A counterexample from \cite{JG} for $r=2$
shows that \eqref{eq:42} holds with sharp ranges of exponents.  This
counterexample plays an important role showing $r$-variation estimates only 
hold when $r>2$. In fact this is the best we can expect in
applications in analysis and ergodic theory.

Inequality \eqref{eq:42} can be thought of as an extension of Doob's
maximal inequality for martingales, which gives a quantitative form of
the martingale convergence theorem. Indeed on the one hand, inequality
\eqref{eq:42} implies that the sequences $(\mathfrak f_n: n\in\ZZ_+)$
converges almost everywhere on $X$ as $n\to \infty$. One the other
hand, one has 
\begin{align*}
\norm{\sup_{n\in\ZZ_+}|\mathfrak f_n|}_{L^p(X)}\le \norm{V^r(\mathfrak f_n: n\in\ZZ_+)}_{L^p(X)}+\norm{\mathfrak f_{n_0}}_{L^p(X)}
\end{align*}
for any $n_0\in\ZZ_+$ (see \eqref{eq:68} below),
which shows that $r$-variational estimates lie deeper than
maximal function estimates.  We refer to \cite{PX,B3,MSZ1}
for generalizations and different proofs of \eqref{eq:42}.

Interestingly, Bourgain \cite{B3} gave a new proof of inequality \eqref{eq:42},
where it was used to address the issue of pointwise convergence of
$A_{M;X, T}^{P}f$, see \eqref{eq:8}. This initiated
a systematic study of $r$-variations and other quantitative estimates
in harmonic analysis and ergodic theory, which resulted in a vast
literature \cite{jkrw,JR1,JSW,JG, OSTTW, MST2, MSZ1, MSZ2, MSZ3, zk} and recently \cite{IMMS, KMT,  MSS}.
Due to \eqref{eq:72} below one has
\begin{align}
\label{eq:43}
\sup_{\lambda>0}\|\lambda N_{\lambda}(\mathfrak{f}_n: n\in\ZZ_+)^{1/r}\|_{L^p(X)}
\le \|V^r(\mathfrak{f}_n: n\in\ZZ_+)\|_{L^p(X)},
\end{align}
which combined with \eqref{eq:42} implies $\lambda$-jump inequalities for
martingales for any $r>2$. Although the right hand side of
\eqref{eq:43} blows up when $r\to 2$, it is possible to prove that for
every $p\in(1, \infty)$ there exists a constant $C_p>0$ such that
\begin{align}
\label{eq:44}
\sup_{\lambda>0}\|\lambda N_{\lambda}(\mathfrak{f}_n: n\in\ZZ_+)^{1/2}\|_{L^p(X)}
\le C_p\sup_{n\in\ZZ_+}\norm{\mathfrak f_n}_{L^p(X)}.
\end{align}
Inequality \eqref{eq:44} was first established by Pisier and Xu
\cite{PX} on $L^2(X)$ and then extended by Bourgain \cite[inequality
(3.5)]{B3} on $L^p(X)$ for all $p\in(1, \infty)$. In fact, Bourgain
used \eqref{eq:44} to prove \eqref{eq:42} by noting that \eqref{eq:43}
can be reversed in the sense that for every $p\in[1,\infty]$ and
$1\le \rho<r\le\infty$ one has
\begin{align}
\label{eq:45}
\| V^{r}( \mathfrak{f}_n : n \in \ZZ_+ ) \|_{L^{p,\infty} (X)}
\lesssim_{p,\rho,r}
\sup_{\lambda>0} \| \lambda N_{\lambda}(\mathfrak{f}_n: n\in \ZZ_+)^{1/\rho} \|_{L^{p,\infty}(X)},
\end{align}
which follows from \eqref{estt1} below. One cannot replace
$L^{p,\8} (X)$ with $L^{p} (X)$ in \eqref{eq:45}, see \cite{MSS} for
more details. Combining \eqref{eq:44} and \eqref{eq:45} with
$\rho=2$ and interpolating, one obtains
\eqref{eq:42}. Therefore uniform $\lambda$-jump estimates from \eqref{eq:44}
can be thought of as endpoint estimates for $r$-variations where we have  seen that
$r$-variations may be unbounded at the endpoint in
question. We have already noted the failure of L{\'e}pingle's inequality
\eqref{eq:42} when $r=2$.

Even though we have a fairly complete picture of the relationship between
$r$-variations and $\lambda$-jumps, the relations with
$r$-oscillations are less obvious. It follows from \eqref{eq:62} below that
\begin{align}
\label{eq:46}
\sup_{J\in\ZZ_+}\sup_{I\in\mathfrak{S}_{J}(\ZZ_+)}\|O^{r}_{I, J}( \mathfrak{f}_n : n \in \ZZ_+ ) \|_{L^{p} (X)}
\le \| V^{r}( \mathfrak{f}_n : n \in \ZZ_+ ) \|_{L^{p} (X)},
\end{align}
where $\mathfrak{S}_{J}(\ZZ_+)$ denotes the set of all strictly
increasing sequences $I=(I_j:j\in\NN_{\le J})\subset \ZZ_+$ of length
$J+1$ for some $J\in\ZZ_+$.  In view of
\eqref{eq:42}, this immediately implies $r$-oscillations estimates for martingales on $L^p(X)$ for all 
 $r\in(2, \infty)$ and
$p\in(1, \infty)$.

It was shown by Jones, Kaufman, Rosenblatt and Wierdl \cite[Theorem 6.4, p. 930]{jkrw} that for every $p\in(1, \infty)$ there is a constant $C_p>0$ such that
\begin{align}
\label{eq:47}
\sup_{J\in\ZZ_+}\sup_{I\in\mathfrak{S}_{J}(\ZZ_+)}\|O^{2}_{I, J}( \mathfrak{f}_n : n \in \ZZ_+ ) \|_{L^{p} (X)}
\le C_{p}\sup_{n\in\ZZ_+}\norm{\mathfrak f_n}_{L^p(X)}.
\end{align}
Inequality \eqref{eq:47} is also  an extension of Doob's
maximal inequality for martingales, as one has
\begin{align*}
\norm{\sup_{n\in\ZZ_+}|\mathfrak f_n|}_{L^p(X)}\le \sup_{J\in\ZZ_+}\sup_{I\in\mathfrak{S}_{J}(\ZZ_+)}\|O^{2}_{I, J}( \mathfrak{f}_n : n \in \ZZ_+ ) \|_{L^{p} (X)}+\norm{\mathfrak f_{n_0}}_{L^p(X)}
\end{align*}
for any $n_0\in\ZZ_+$. This follows from Proposition
\ref{prop:5} below. Moreover, in view of Proposition \ref{prop:4}, inequality
\eqref{eq:47} also gives a quantitative form of the martingale
convergence theorem.

In Section \ref{section:3} we give a new proof of inequality
\eqref{eq:47}, which follows from an abstract result formulated for
certain projections, see Theorem \ref{thm:p} in Section \ref{section:3}. This abstract theorem will also
establish oscillation inequalities for smooth bump functions
(see Proposition \ref{prop:2} and Theorem \ref{thm:1}), and establish
oscillation inequalities for the Carleson operator (see Proposition
\ref{prop:car} as well as Proposition \ref{prop:6}). It will also show that
oscillation estimates are very close to maximal estimates even though
it follows from Proposition \ref{prop:5} that oscillations always
dominate maximal functions, see the discussion below Theorem \ref{bourgain}.

Inequalities \eqref{eq:46} and \eqref{eq:47} are  similar
to inequalities \eqref{eq:43} and \eqref{eq:44}, respectively, and this raises a
natural question whether $2$-oscillations can be interpreted as an
endpoint for $r$-variations when $r>2$ in the sense of inequality
\eqref{eq:45}. Recently this problem was investigated in \cite[Theorem 1.9]{MSS} and answered in the
negative. Specifically, one can show if
$1 \le p < \infty$ and
$1 < \rho \le r < \infty$ are fixed, then it is \textbf{not true} that the
estimates
\begin{align}
\label{eq:49}
\begin{split}
\sup_{\lambda > 0} \| \lambda N_{\lambda}(f(\cdot,t):t\in\NN)^{1/r} \|_{\ell^{p,\infty}(\ZZ)} 
&\le
C_{p,\rho,r}
\sup_{I\in\mathfrak S_\infty (\NN)} 
\| O_{I,\infty}^{\rho}(f(\cdot,t):t\in\NN) \|_{\ell^{p}(\ZZ)},\\
\| V^r(f(\cdot,t):t\in\NN) \|_{\ell^{p,\infty}(\ZZ)} 
&\le
C_{p,\rho,r}
\sup_{I\in\mathfrak S_\infty (\NN)} 
\| O_{I,\infty}^{\rho} (f(\cdot,t):t\in\NN) \|_{\ell^{p}(\ZZ)}
\end{split}
\end{align}
hold uniformly for every measurable function
$f \colon \ZZ \times \NN \to \RR$.  The failure of the inequalities \eqref{eq:49} 
shows that the space induced by $\rho$-oscillations is different from
the spaces induced by $r$-variations and $\lambda$ jumps whenever
$\rho\le r$.  Also the failure of the inequalities \eqref{eq:49} shows that $\rho$-oscillation
inequalities cannot be seen (at least in a straightforward way,
understood in the sense of inequality \eqref{eq:45}) as endpoint
estimates for $r$-variations, though it still makes sense to ask
whether a priori bounds for $2$-oscillations imply bounds for
$r$-variations for any $r>2$. This is an intriguing question from the
point of view of quantitative pointwise convergence problems. If true, it
would reduce pointwise convergence problems to the study of $2$-oscillations,
which in certain cases are simpler since they are closer to square
functions.

\subsection{Quantitative forms of Bourgain's ergodic theorem}
Quantitative bounds in the context of ergodic  polynomial averaging
operators have been intensively studied over the last decade. These investigations were the subject of the following papers \cite{MST1, MST2, MSZ3, MSS}, which generalized Bourgain's papers \cite{B1, B2, B3} in various ways and can be summarized as follows.

\begin{theorem}
\label{thm:mainet}
Let $d, k\in\ZZ_+$, and
$\mathcal P=(P_1, \ldots, P_d)\subset\ZZ[{\rm m}_1, \ldots, {\rm m}_k]$
such that $P_j(0)=0$ for $j\in[d]$ be given.  Let $(X, \calB(X), \mu)$
be a $\sigma$-finite measure space endowed with a family
$\calT=(T_1,\ldots, T_d)$ of commuting invertible measure-preserving
transformations on $X$.  Let $f\in L^p(X)$ for some
$1\le p\le \infty$, and for $M\in\ZZ_+$ let
$A_{{M}; X, {\mathcal T}}^{\mathcal P}f = A_{M_1,\ldots, M_k;X,T_1,\ldots, T_d}^{P_1,\ldots, P_d} f$
be the polynomial ergodic average defined in \eqref{eq:31} with the
parameters $M_1=\ldots=M_k=M$.
\begin{itemize}
\item[(i)] \textit{(Mean ergodic theorem)} If $1<p<\infty$, then the averages
$A_{{M}; X, {\mathcal T}}^{\mathcal P}f$ converge in $L^p(X)$ norm as $M\to \infty$.

\item[(ii)] \textit{(Pointwise ergodic theorem)} If $1<p<\infty$, then the averages
$A_{{M}; X, {\mathcal T}}^{\mathcal P}f$ converge pointwise almost everywhere on $X$ as $M\to\infty$.

\item[(iii)] \textit{(Maximal ergodic theorem)}
If $1<p\le\infty$, then one has
\begin{align}
\label{eq:5}
\big\|\sup_{M\in\ZZ_+}|A_{{M}; X, {\mathcal T}}^{\mathcal P}f|\big\|_{L^p(X)}\lesssim_{d,k,p, \deg \mathcal P}\|f\|_{L^p(X)}.
\end{align}

\item[(iv)] \textit{(Variational ergodic theorem)}
If $1<p <\infty$ and $2<r<\infty$, then one has
\begin{align}
\label{eq:7}
\big\|V^r(A_{{M}; X, {\mathcal T}}^{\mathcal P}f: M\in\ZZ_+)\big\|_{L^p(X)}\lesssim_{d,k,p, r,  \deg \mathcal P}\|f\|_{L^p(X)}.
\end{align}

\item[(v)] \textit{(Jump ergodic theorem)}
If $1<p<\infty$, then one has
\begin{align}
\label{eq:36}
\sup_{\lambda>0}\big\|\lambda N_{\lambda}(A_{{M}; X, {\mathcal T}}^{\mathcal P}f:M\in\ZZ_+)^{1/2}\big\|_{L^p(X)}\lesssim_{d,k,p, \deg \mathcal P}\|f\|_{L^p(X)}.
\end{align}

\item[(vi)] \textit{(Oscillation ergodic theorem)}
If $1<p<\infty$,  then one has
\begin{align}
\label{eq:37}
\sup_{J\in\ZZ_+}\sup_{I\in\mathfrak S_J(\ZZ_+) }\big\|O_{I, J}^2(A_{{M}; X, {\mathcal T}}^{\mathcal P}f: M\in\ZZ_+)\big\|_{L^p(X)}\lesssim_{d,k,p, \deg \mathcal P}\|f\|_{L^p(X)}.
\end{align}
\end{itemize}
Moreover, the implicit constants in \eqref{eq:5}, \eqref{eq:7}, \eqref{eq:36} and \eqref{eq:37} can be taken to be independent of the
coefficients of the polynomials from $\mathcal P$, depending only on $p$ and the degree of the family $\mathcal P$.
\end{theorem}

We now give some remarks about Theorem \ref{thm:mainet}.

\begin{enumerate}[label*={\arabic*}.]

\item Theorem \ref{thm:mainet} is a multi-dimensional, quantitative
counterpart of Theorem \ref{bourgain} with sharp ranges of parameters
$1<p<\infty$ and $2<r<\infty$, which contributes to the
Furstenberg--Bergelson--Leibman conjecture \cite[Section 5.5,
p. 468]{BL} in the linear case for the class of commuting
measure-preserving transformations. The
Furstenberg--Bergelson--Leibman conjecture is a central open problem in pointwise ergodic
theory. Moreover, inequalities \eqref{eq:36} and \eqref{eq:37} are the
strongest possible quantitative forms of pointwise convergence.  
By taking $d=k=1$ and $P_1(m)=m$ in Theorem
\ref{thm:mainet}, we recover Birkhoff's and von Neumann's results
stated in Theorem \ref{birk}. Taking $d=k=1$ and $P_1\in\ZZ[{\rm m}]$
in Theorem \ref{thm:mainet}, we also recover Bourgain's polynomial
ergodic theorem from Theorem \ref{bourgain} above.

\item The mean ergodic theorem in (i) is a consequence of the dominated
convergence theorem combined with (ii) and (iii).  Each of the
conclusions from (iv), (v) and (vi) individually implies pointwise
convergence from (ii), as well as the maximal estimates from (iii). It
also follows from \eqref{estt1} that (v) implies (iv).  
Details about these implications can be easily derived from the properties
of oscillations, variations and jumps collected in Section
\ref{section:2}.

\item Sharp $r$-variational estimates \eqref{eq:7}
were obtained  for the first time in \cite{MST2}, with a conceptually new proof which
also works for other discrete operators with arithmetic features
\cite{MST1}. Not long afterwards, the ideas from \cite{MST2} were 
extended \cite{MSZ3} to establish uniform $\lambda$-jump estimates \eqref{eq:36}. Partial result for
$r$-variational estimates \eqref{eq:7} were obtained in \cite{Kr, zk, MT}.

\item It was observed in \cite{MST2} that \eqref{eq:7} and H\"{o}lder's inequality imply that 
for every
$p\in(1, \infty)$, for any $r>2$, every $f\in L^p(X)$ and every $J\in\ZZ_+$ one has
\begin{align*}
\qquad \qquad \sup_{I\in\mathfrak S_J(\ZZ_+) }\big\|O_{I, J}^2(A_{{M}; X, {\mathcal T}}^{\mathcal P}f: M\in\ZZ_+)\big\|_{L^p(X)}
\lesssim_{d,k,p, r, \deg \mathcal P} J^{\frac{1}{2}-\frac{1}{r}} \norm{f}_{L^p(X)},
\end{align*}
with the same implicit constant as in \eqref{eq:7} and so blows up as
$r$ tends to 2.  This inequality is a non-uniform version of
\eqref{eq:37} in the spirit of Bourgain's oscillation inequality
\eqref{eq:34}. However it was observed recently \cite{MSS} that
the methods from \cite{MST2, MSZ3} give the uniform oscillation
inequality in \eqref{eq:37}.  From this point of view (and from the
discussion above for martingales) inequality
\eqref{eq:37} can be thought of as an endpoint for \eqref{eq:7} at
$r=2$, though it is not an endpoint in the sense of inequality \eqref{estt1} below. It
would be nice to know whether it is possible (if at all) to use \eqref{eq:37} to recover \eqref{eq:7}.

\item Inequality \eqref{eq:37} is also a contribution to an interesting
problem from the early 1990's of Rosenblatt and Wierdl \cite[Problem
4.12, p. 80]{RW} about uniform estimates of oscillation inequalities
for ergodic averages. In \cite{jkrw} Jones, Kaufman, Rosenblatt and
Wierdl proved \eqref{eq:37} for the classical Birkhoff averages with
$d=k=1$ and $P_1(m)=m$, giving an affirmative answer to \cite[Problem
4.12, p. 80]{RW}. In \cite{MSS} it was shown that \cite[Problem 4.12,
p. 80]{RW} remains true  even for multidimensional polynomial
ergodic averages.

\item The proof of Theorem \ref{thm:mainet} is an elaboration of
methods developed in \cite{MST2, MSZ3} and also recently in
\cite{MSS}.  The main tools are the Hardy--Littlewood circle method
(major arcs estimates); Weyl's inequality (minor arcs estimates); the
Ionescu--Wainger multiplier theory (see \cite{IW, M10, MSZ3} and also \cite{Pierce},
\cite{TaoIW}); the Rademacher--Menshov argument (see for instance
\cite{MSZ2}) and the sampling principle of Magyar--Stein--Wainger (see
\cite{MSW} and also \cite{MSZ1}). The methods from \cite{M10, MST2,
MSZ1, MSZ2, MSZ3} were further developed by the first author in
collaboration with Krause and Tao \cite{KMT}, which resulted in
establishing pointwise convergence for the so-called bilinear
Furstenberg--Weiss ergodic averages.  This was a long-standing open
problem, which makes a significant contribution towards the Furstenberg--Bergelson--Leibman
conjecture \cite{BL}.

\end{enumerate}

\subsection{A multi-parameter variant of the Bellow and Furstenberg problem}
After completing \cite{B1, B2, B3}, Bourgain observed that the
Dunford--Zygmund theorem (see Theorem \ref{dz}) can be extended to the
polynomial setting at the expense of imposing that the 
measure-preserving transformations in Theorem
\ref{dz} commute. Bourgain's result can be formulated as follows. 

\begin{theorem}[Polynomial Dunford--Zygmund ergodic theorem]
\label{thm:main}
Let $d\in\ZZ_+$, and
$P_1, \ldots, P_d\in\ZZ[{\rm m}]$
such that $P_j(0)=0$ for $j\in[d]$ be given.  Let $(X, \calB(X), \mu)$
be a $\sigma$-finite measure space endowed with a family
$\calT=(T_1,\ldots, T_d)$ of commuting invertible measure-preserving
transformations on $X$.  Let $f\in L^p(X)$ for some
$1\le p\le \infty$, and for $M\in\ZZ_+^d$ let
$A_{M; X, \calT}^{P_1 ({\mathrm m}_1),\ldots, P_d ({\mathrm m}_d)}f = A_{M_1,\ldots, M_d;X,T_1,\ldots, T_d}^{P_1 ({\mathrm m}_1),\ldots, P_d ({\mathrm m}_d)}f$
be the polynomial ergodic  average defined  in \eqref{eq:31}.
\begin{itemize}
\item[(i)] \textit{(Mean ergodic theorem)} If $1<p<\infty$, then the averages
$A_{M; X, \calT}^{P_1 ({\mathrm m}_1),\ldots, P_d ({\mathrm m}_d)}f$ converge in $L^p(X)$ norm as $\min\{M_1,\ldots, M_d \}\to\infty$.

\item[(ii)] \textit{(Pointwise ergodic theorem)} If $1<p<\infty$, then the averages
$A_{M; X, \calT}^{P_1 ({\mathrm m}_1),\ldots, P_d ({\mathrm m}_d)}f$ converge pointwise almost everywhere on $X$ as $\min\{M_1,\ldots, M_d\}\to\infty$.

\item[(iii)] \textit{(Maximal ergodic theorem)}
If $1<p\le\infty$, then one has
\begin{align}
\label{eq:51}
\big\|\sup_{M\in\ZZ_+^d}|A_{M; X, \calT}^{P_1 ({\mathrm m}_1),\ldots, P_d ({\mathrm m}_d)}f|\big\|_{L^p(X)}\lesssim_{d,p, \deg P_1,\ldots, \deg P_d}\|f\|_{L^p(X)}.
\end{align}

\item[(iv)] \textit{(Oscillation ergodic theorem)}
If $1<p<\infty$,  then one has
\begin{align}
\label{eq:52}
\sup_{J\in\ZZ_+}\sup_{I\in\mathfrak S_J(\ZZ_+^d) }\big\|O_{I, J}^2(A_{M; X, \calT}^{P_1 ({\mathrm m}_1),\ldots, P_d ({\mathrm m}_d)}f: M\in\ZZ_+^d)\big\|_{L^p(X)}\lesssim_{d,p, \deg P_1,\ldots, \deg P_d}\|f\|_{L^p(X)}.
\end{align}
\end{itemize}
We refer to Section \ref{section:2} for the definitions of the sets $\mathfrak S_J(\ZZ_+^d)$, see \eqref{eq:2}, and the multi-parameter oscillations, see \eqref{eq:102}.
Moreover, the implicit constants in \eqref{eq:51} and \eqref{eq:52} can be taken to be independent of the
coefficients of the polynomials $P_1,\ldots, P_d$, depending only on $p$ and $\deg P_1,\ldots, \deg P_d$.
\end{theorem}

We now give some remarks about Theorem \ref{thm:main}.
\begin{enumerate}[label*={\arabic*}.]

\item Theorem \ref{thm:main}(i)-(iii) is attributed to Bourgain,
though it has never been published. The first and third authors
learned about this result from Bourgain in October 2016, when they started
to work with Bourgain and Stein on some aspects of multi-parameter
ergodic theory \cite{BMSWer}.

\item In this paper we prove Theorem \ref{thm:main} using a general abstract principle, see Proposition \ref{prop:8}
in Section \ref{section:4}. In contrast to Bourgain's original
observation our proof of Theorem \ref{thm:main} relies on uniform bounds for multi-parameter
oscillation inequalities.

\item Theorem \ref{thm:main}(iv) with linear polynomials
$P_1(m)=\ldots=P_d(m)=m$ was established in \cite{JRW}, where it was
essential that $\calT=(T_1,\ldots, T_d)$ is a commuting family of
measure-preserving transformations on $X$. It is straightforward to see that
(iv) implies (iii) by \eqref{eq:9}, as well as (ii) by appealing to
Proposition \ref{prop:4}. Using the dominated convergence theorem with
(ii) and (iii) we also obtain (i). So it suffices to prove
\eqref{eq:52}, which we do in Section \ref{section:4}.

\item To prove Theorem \ref{thm:main} it is essential to note that
\begin{align}
\label{eq:50}
A_{M; X, \calT}^{P_1 ({\mathrm m}_1),\ldots, P_d ({\mathrm m}_d)}f=A_{M_1,\ldots, M_d; X, T_1,\ldots, T_d}^{P_1 ({\mathrm m}_1),\ldots, P_d ({\mathrm m}_d)}f=A_{M_1; X, T_1}^{P_1({\mathrm m}_1)}\circ\ldots\circ A_{M_d; X, T_d}^{P_d({\mathrm m}_d)}f,
\end{align}
where the latter averages (defined as in \eqref{eq:8}) commute as long
as the family $\calT=(T_1,\ldots, T_d)$ is commuting. Using identity
\eqref{eq:50} and iterating appropriately \eqref{eq:37} with $k=d=1$
we will be able to derive \eqref{eq:52}. We refer to Section
\ref{section:4} for details.
\end{enumerate}

Theorem \ref{thm:main} can be thought of as a simple case of a multi-parameter variant of 
the Bellow and Furstenberg problem, which is a central open problem in modern ergodic theory
and  can be subsumed under the following conjecture:
\begin{conjecture}
\label{con:1}
Let $d, k\in\ZZ_+$ be given and let $(X, \mathcal B(X), \mu)$ be a probability measure space endowed with a family $\calT=(T_1,\ldots, T_d)$  of invertible commuting  measure-preserving transformations on $X$. Assume that $\calP=(P_1,\ldots, P_d) \subset\ZZ[{\rm m}_1, \ldots, {\rm m}_k]$ such that $P_j(0)=0$ for $j\in[d]$ are given. Then for any $f\in L^{\infty}(X)$ the multi-parameter polynomial averages $A_{M; X, \calT}^{\calP}f(x)=A_{M_1,\ldots, M_k;X,T_1,\ldots, T_d}^{P_1,\ldots, P_d} f(x)$ defined in \eqref{eq:31} converge for $\mu$-almost every $x\in X$, as  $\min\{M_1,\ldots, M_k\}\to\infty$.
\end{conjecture}

A few remarks about this conjecture, its history, and the current state of the art are in order.

\begin{enumerate}[label*={\arabic*}.]
\item As we have seen above the case $d=k=1$ of Conjecture \ref{con:1} with $P_1(m)=m$ follows from Birkhoff's ergodic theorem, see Theorem \ref{birk}.  The case $d=k=1$  of Conjecture \ref{con:1} with arbitrary polynomials $P_1\in\ZZ[\rm n]$
was the famous open problem of Bellow \cite{Bel} and Furstenberg \cite{F} and solved by Bourgain \cite{B1, B2, B3} in the mid 1980's, see Theorem \ref{bourgain}. The general case $d,k\in\ZZ_+$ of Conjecture \ref{con:1} with arbitrary polynomials $P_1,\ldots, P_d\in\ZZ[{\rm m}_1, \ldots, {\rm m}_k]$ in the diagonal setting $M_1=\ldots=M_k$, that is, the multi-dimensional one-parameter setting, follows from Theorem \ref{thm:mainet}.

\item A genuinely multi-parameter  case $d=k\ge2$  of Conjecture \ref{con:1} for averages \eqref{eq:31} with $P_{j}(m_1, \ldots, m_{d})=P_j(m_j)$, where $P_j\in\ZZ[{\rm m}_j]$ for $j\in[d]$ follows from Theorem \ref{thm:main},
which extends the case of linear polynomials $P_1(m)=\ldots=P_d(m)=m$  established independently by Dunford \cite{D} and Zygmund \cite{Z} in the early 1950's, see Theorem \ref{dz}.  

\item Thanks to the product structure of \eqref{eq:50} Theorem~\ref{dz}, as well as Theorem~\ref{thm:main}, have relatively simple one-parameter
proofs, which are based on iterative applications of Theorem~\ref{birk}
and Theorem~\ref{thm:mainet}, respectively. This is explained in
Proposition \ref{prop:8} below.  However, the
situation is dramatically different when orbits in \eqref{eq:31}
are defined along genuinely $k$-variate polynomials
$P_1,\ldots, P_d\in\ZZ[{\rm m}_1, \ldots, {\rm m}_k]$ since then we lose
the product structure \eqref{eq:50}. This can be illustrated by
considering averages \eqref{eq:31} for $d=1$, $k=2$ with, let us say,
$P_1(m_1, m_2)=m_1^2m_2^3$. Then Conjecture \ref{con:1} becomes 
challenging. Surprisingly, even in this simple case, it seems that there is no simple way (like
changing variables or interpreting the average from \eqref{eq:31}
as a composition of simpler one-parameter averages as in
\eqref{eq:50}) that would help us reduce the
matter to the setup where pointwise convergence is known. 
This was one of the motivations leading to Conjecture \ref{con:1}.

\item The Dunford--Zygmund theorem (see Theorem \ref{dz} above) was
originally proved for not necessarily commuting,
measure-preserving transformations $\calT=(T_1,\ldots, T_d)$ on $X$. However, it is well known
for instance from the Bergelson--Leibman paper \cite{BL}  that the 
commutation assumption imposed on the family $T_1,\ldots, T_d:X\to X$
in \eqref{eq:31} is essential in order to have an ergodic theorem if
$\deg P_j\ge2$ for at least one $j\in[d]$ and $d\ge2$. Even in the
one-parameter case (assuming $k=1$) in \eqref{eq:31} an ergodic
theorem may fail. The question to what extent one can relax
commutation relations among $T_1,\ldots, T_d$ in \eqref{eq:31}, even
in the one-parameter case, is very intriguing. This also motivates the desire to understand
Conjecture \ref{con:1} in the commutative setting first, as it is unclear whether 
Conjecture \ref{con:1} is true for all polynomials  $P_1,\ldots, P_d\in\ZZ[{\rm m}_1, \ldots, {\rm m}_k]$.

\item With respect to the noncommutative setting, we
mention that recently the first and second authors with Ionescu and
Magyar \cite{IMMS} established Conjecture \ref{con:1} with $k=1$, 
$d\in\ZZ_+$ and arbitrary polynomials
$P_1,\ldots, P_d\in\ZZ[\rm m]$ in the diagonal
nilpotent setting, i.e. one-parameter and multi-dimensional, when $\calT=(T_1,\ldots, T_d)$ is a family of
invertible measure-preserving transformations of a $\sigma$-finite
measure space $(X, \mathcal B(X), \mu)$ that generates a nilpotent
group of step two. In view of the Bergelson--Leibman paper \cite{BL}, the
nilpotent setting is probably the most general setting where
Conjecture \ref{con:1} might be true, at least in the one-parameter case.

\item We finally mention that progress towards establishing
Conjecture \ref{con:1} was recently made by the first and third
authors in collaboration with Bourgain and Stein \cite{BMSWer}. This
conjecture was verified for any integer $d\ge2$ with $k=d-1$ for
averages \eqref{eq:31} with polynomials
\begin{align}
\label{eq:53}
\begin{split}
P_{j}({\rm m}_1, \ldots, {\rm m}_{d-1})&={\rm m}_j \quad \text{for}\quad  j\in[d-1]; \quad \text{and} \\
P_{d}({\rm m}_1, \ldots, {\rm m}_{d-1})&=P({\rm m}_1, \ldots, {\rm m}_{d-1}),
\end{split}
\end{align}
whenever $P\in\ZZ[{\rm m}_1,\ldots,  {\rm m}_{d-1}]$ is a polynomial such that
\begin{align*}
\qquad P(0,\ldots, 0)=\partial_1P(0,\ldots, 0)=\ldots=\partial_{d-1}P(0,\ldots, 0)=0,
\end{align*}
which has partial degrees (as a polynomial of the variable ${\rm m}_i$
for any $i\in[d-1]$) at least two. Furthermore, it follows from \cite{BMSWer} that for any $P\in\ZZ[{\rm m}_1,\ldots, {\rm m}_{d}]$ the following averages
\begin{align}
\label{eq:3}
\qquad \qquad A_{{M}; X, {\mathcal T}}^{P}f(x): =  \EE_{(m_1,\ldots, m_{d})\in Q_{M}}f(T^{P(m_1,\ldots, m_{d})}x), \qquad x\in X,
\end{align}
where  $M=(M_1,\ldots, M_{d})\in\ZZ_+^{d},$
converge almost everywhere on $X$ as $\min\{M_1,\ldots, M_{d}\}\to\infty$. 
In fact, Conjecture \ref{con:1} was originally formulated with averages \eqref{eq:3}, the authors learned about this from Jean Bourgain in a private communication in October 2016. The proof from \cite{BMSWer}
developed new methods from Fourier analysis and number
theory.  Even though the averages \eqref{eq:31} with polynomials from
\eqref{eq:53} share a lot of difficulties that arise in the general
case, there are some cases that are not covered by the methods
developed in \cite{BMSWer}. At this moment it is not clear whether
Conjecture \ref{con:1} is true in full generality. The work in 
\cite{BMSWer} is a significant step towards understanding
Conjecture \ref{con:1} that sheds new light on the general case and will
either lead to its full resolution or to a counterexample.  The  authors plan to investigate this question in the near future.

\end{enumerate}

\subsection{Overview of the paper}
In this paper we prove an abstract principle for the so-called
projective operators, see Theorem \ref{thm:p} in Section
\ref{section:3}, which allows us to deal with one-parameter
oscillation inequalities in a fairly unified way. As a consequence of
Theorem \ref{thm:p} we give a simple proof of Jones, Kaufman,
Rosenblatt and Wierdl \cite[Theorem 6.4, p. 930]{jkrw} oscillation
inequality for martingales, see Proposition \ref{prop:1}, then we
prove oscillation inequalities for smooth bumps, see Proposition
\ref{prop:2} and Theorem \ref{thm:1}. Further, we discuss oscillation
estimates for projection operators corresponding to orthonormal
systems in Hilbert spaces, see Proposition \ref{prop:6}, and finally we
obtain new oscillations inequalities for the Carleson operator,
see Proposition \ref{prop:car}. In Section \ref{section:4} we
build a multi-parameter theory of oscillation estimates, see
Proposition \ref{prop:8} and Corollary \ref{cor:2}. As an application of our method, we give a simple
proof of Theorem \ref{thm:main}.

This paper can be viewed as a fairly systematic treatment of
oscillation estimates in the one-parameter as well as multi-parameter
settings in ergodic theory and analysis.  In the multi-parameter setting, oscillation semi-norms
seem to be the only viable
tool that allows us to handle efficiently multi-parameter pointwise
convergence problems. This is especially the case in \cite{BMSWer}
where operators with arithmetic features were studied.  It also
contrasts sharply with the one-parameter setting, where we
have a variety of available tools to handle pointwise convergence
problems: including oscillations, variations or jumps, see \cite{JSW,
MST2} and the references given there.

\section{Notation and useful tools}\label{section:2}
We now set some notation that will be used throughout the paper. Basic
properties of one-parameter as well as multi-parameter $r$-oscillation
semi-norms, $r$-variation semi-norms and $\lambda$-jump counting functions will
be also gathered here. We borrow notation from \cite[Section 2]{BMSWer} and \cite[Section 2]{MSS}.

\subsection{Basic notation}
Let $\ZZ_+:=\{1, 2, \ldots\}$, $\NN:=\{0,1,2,\ldots\}$ and
$\RR_+:=(0, \infty)$.  For $d\in\ZZ_+$ the sets $\ZZ^d$, $\RR^d$,
$\CC^d$ and $\TT^d:=\RR^d/\ZZ^d$ have standard meaning. We will also
consider the set of dyadic numbers $\DD:=\{2^n: n\in\ZZ\}$.  For any
$x\in\RR$ we define the floor function
\[
\lfloor x \rfloor: = \max\{ n \in \ZZ : n \le x \}.
\]
For $x, y\in\RR$ let $x \wedge y := \min\{x,y\}$ and $x \vee y := \max\{x,y\}$.  
For every $N\in\RR_+$ and $\AA\subseteq\RR$  define
\[
[N]:=(0, N]\cap\ZZ=\{1, \ldots, \lfloor N\rfloor\},
\]
as well as
\begin{align*}
\AA_{\le N}:= [0, N]\cap\AA,\ \: \quad &\text{ and } \quad
\AA_{< N}:= [0, N)\cap\AA,\\
\AA_{\ge N}:= [N, \infty)\cap\AA, \quad &\text{ and } \quad
\AA_{> N}:= (N, \infty)\cap\AA.
\end{align*}

We use $\ind{A}$ to denote the indicator function of a set $A$. If $S$ 
is a statement we write $\ind{S}$ to denote its indicator, equal to $1$
if $S$ is true and $0$ if $S$ is false. For instance $\ind{A}(x)=\ind{x\in A}$.

For two nonnegative quantities $A, B$ we write $A \lesssim B$ if there
is an absolute constant $C>0$ such that $A\le CB$, however $C>0$ may
change from occurrence to occurrence. We will write $A \simeq B$ when
$A \lesssim B\lesssim A$.  We will write $\lesssim_{\delta}$ or
$\simeq_{\delta}$ to emphasize that the implicit constant depends on
$\delta$. For two functions $f:X\to \CC$ and $g:X\to [0, \infty)$,
write $f = O(g)$ if there exists $C>0$ such that $|f(x)| \le C g(x)$
for all $x\in X$. We will also write $f = O_{\delta}(g)$ if the
implicit constant depends on $\delta$.

\subsection{Euclidean spaces} The standard inner product, the
corresponding Euclidean norm, and the maximum norm on $\RR^d$ are
denoted respectively, for any $x=(x_1,\ldots, x_d)$,
$\xi=(\xi_1, \ldots, \xi_d)\in\RR^d$, by
\begin{align*}
x\cdot\xi:=\sum_{k=1}^dx_k\xi_k, \qquad \text{ and } \qquad
\abs{x}:=\abs{x}_2:=\sqrt{\ipr{x}{x}}, \qquad \text{ and } \qquad |x|_{\infty}:=\max_{k\in[d]}|x_k|.
\end{align*}

\subsection{Function spaces}
Throughout this paper all vector spaces will be defined over $\CC$.
For a continuous linear map $T : B_1 \to B_2$ between two normed
vector spaces $B_1$ and $B_2$, its operator norm will be denoted by
$\|T\|_{B_1 \to B_2}$.

The triple $(X, \mathcal B(X), \mu)$ denotes a measure space $X$ with
a $\sigma$-algebra $\mathcal B(X)$ and a $\sigma$-finite measure
$\mu$.  The space of all $\mu$-measurable functions $f:X\to\CC$ will
be denoted by $L^0(X)$.  The space of all functions in $L^0(X)$ whose
modulus is integrable with $p$-th power is denoted by $L^p(X)$ for
$p\in(0, \infty)$, whereas $L^{\infty}(X)$ denotes the space of all
essentially bounded functions in $L^0(X)$.  These notions can be
extended to functions taking values in a separable normed vector space
$(B, \|\cdot\|_B)$, for instance
\begin{align*}
L^{p}(X;B)
:=\big\{F\in L^0(X;B):\|F\|_{L^{p}(X;B)} \coloneqq \left\|\|F\|_B\right\|_{L^{p}(X)}<\infty\big\},
\end{align*}
where $L^0(X;B)$ denotes\footnote{Note that there are various definitions of $L^0(X;B)$ in the literature.} the space of measurable functions from $X$ to
$B$ (up to almost everywhere equivalence). 
For any $p\in[1,\infty]$ we define a weak-$L^p$ space of measurable functions on $X$ by setting
\begin{equation*}
    L^{p,\infty}(X):=\{f:X\to\CC\colon\norm{f}_{L^{p,\infty}(X)}<\infty\},
\end{equation*}
where for any $p\in[1,\infty)$ we have
\begin{align*}
\norm{f}_{L^{p,\infty}(X)}:=\sup_{\lambda>0}\lambda\mu(\{x\in X:|f(x)|>\lambda\})^{1/p},
\qquad\text{\and}\qquad
\norm{f}_{L^{\infty,\infty}(X)}:=\norm{f}_{L^{\infty}(X)}.
\end{align*}

In our case we will mainly take  $X=\RR^d$ or
$X=\TT^d$ equipped with the Lebesgue measure, and   $X=\ZZ^d$ endowed with the
counting measure. If $X$ is endowed with a counting measure we will
abbreviate $L^p(X)$ to $\ell^p(X)$ and $L^p(X; B)$ to $\ell^p(X; B)$ and $L^{p,\infty}(X)$ to $\ell^{p,\infty}(X)$.

\subsection{Fourier transform}  
We will use the convention that $\ex(z)=e^{2\pi {\bm i} z}$ for
every $z\in\CC$, where ${\bm i}^2=-1$. Let $\calF_{\RR^d}$ denote the Fourier transform on $\RR^d$ defined for
any $f \in L^1(\RR^d)$ and for any $\xi\in\RR^d$ as
\begin{align*}
\calF_{\RR^d} f(\xi) := \int_{\RR^d} f(x) \ex(x\cdot\xi) {\rm d}x.
\end{align*}
We can also consider the Fourier transform for finite Borel measures $\sigma$ on $\RR^d$. 
If $f \in \ell^1(\ZZ^d)$ we define the discrete Fourier
transform (Fourier series) $\calF_{\ZZ^d}$, for any $\xi\in \TT^d$, by setting
\begin{align*}
\calF_{\ZZ^d}f(\xi): = \sum_{x \in \ZZ^d} f(x) \ex(x\cdot\xi).
\end{align*}
Sometimes we shall abbreviate $\calF_{\ZZ^d}f$ or $\calF_{\RR^d}f$ to $\hat{f}$, if the context will be clear.

Let $\GG=\RR^d$ or $\GG=\ZZ^d$. It is well known that their corresponding dual groups are $\GG^*=(\RR^d)^*=\RR^d$ or $\GG^*=(\ZZ^d)^*=\TT^d$ respectively.
For any bounded function $\mathfrak m: \GG^*\to\CC$ and a test function $f:\GG\to\CC$ we define the Fourier multiplier operator  by 
\begin{align}
\label{eq:1}
T_{\GG}[\mathfrak m]f(x):=\int_{\GG^*}\ex(-\xi\cdot x)\mathfrak m(\xi)\calF_{\GG}f(\xi){\rm d}\xi, \quad \text{ for } \quad x\in\GG.
\end{align}
One may think that $f:\GG\to\CC$ is a compactly supported function on $\GG$ (and smooth if $\GG=\RR^d$) or any other function for which \eqref{eq:1} makes sense.

\subsection{Littlewood--Paley theory}\label{lp} Often we will control oscillation and variation semi-norms by certain
square functions of the form
\[
S(f)(x) := \Bigl(\sum_{k\in {\mathbb Z}} |\sigma_k*f(x)|^2 \Bigr)^{1/2},
\]
where $(\sigma_k)_{k \in \ZZ}$ is a sequence of Borel measures on ${\mathbb R}^d$ with bounded total variation
satisfying $|{\widehat{\sigma_k}}(\xi)| \le C \min\{|a_{k+1} \xi|^{\alpha}, |a_k \xi|^{-\alpha}\}$ for some
$\alpha>0$ and all $k\in \ZZ$. Here $\inf_{k\in\ZZ} a_{k+1}/a_k > 1$.
What we call {\it standard Littlewood--Paley arguments} sometimes refer to the arguments developed
in the seminal paper \cite{DR}. In particular, Theorem B in \cite{DR} implies that the square function $S$
satisfies $L^p$ bounds $\|S(f)\|_{L^p} \le C_p \|f\|_{L^p}$ for all $p\in(1, \infty)$ whenever the corresponding maximal
function $\sigma^{*}$ associated to the measures $(\sigma_k)_{k \in \ZZ}$ satisfies the same $L^p$ bounds.

At one point we will use a powerful square function bound of Rubio de Francia associated to any pairwise disjoint
collection of
intervals $(I_j : j \in \ZZ)$ on $\RR$. It states
$$
\Big\|\Bigl(\sum_{j\in\ZZ} |T_{\RR}[\ind{I_j}] f|^2 \Bigr)^{1/2} \Big\|_{L^p({\RR})} \ \lesssim \ \|f\|_{L^p(\RR)}
$$
whenever $p\in [2, \infty)$. See \cite[Theorem 1.2]{RdF}.

\subsection{Coordinatewise order $\preceq$} For any
$x=(x_1,\ldots, x_k)\in\RR^k$ and $y=(y_1,\ldots, y_k)\in\RR^k$ we say
$x\preceq y$ if an only if $x_i\le y_i$ for each $i\in[k]$.  We also
write $x\prec y$ if an only if $x\preceq y$ and $x\neq y$, and
$x\prec_{\rm s} y$ if an only if $x_i< y_i$ for each $i\in[k]$. Let
$\II\subseteq \RR^k$ be an index set such that $\# \II\ge2$ and for
every $J\in\ZZ_+\cup\{\infty\}$ define the set
\begin{align}
\label{eq:2}
\mathfrak S_J(\II):
=
\Set[\big]{(t_i:i\in\NN_{\le J})\subseteq \II}{t_{0}\prec_{\rm s}
t_{1}\prec_{\rm s}\ldots \prec_{\rm s}t_{J}},
\end{align}
where $\NN_{\le \infty}:=\NN$.
In other
words, $\mathfrak S_J(\II)$ is the family of all strictly increasing
sequences (with respect to the coordinatewise order) of length $J+1$
taking their values in the set $\II$.

\subsection{Oscillation semi-norms}
Let $\II\subseteq \RR^k$ be an index set such that $\#{\II}\ge2$. Let $(\mathfrak a_{t}(x): t\in\II)$ be a $k$-parameter
family of complex-valued measurable functions defined on $X$. For any
$\JJ\subseteq \II$, any $1\le r<\infty$ and a sequence
$I=(I_i : i\in\NN_{\le J}) \in \mathfrak S_J(\II)$ the multi-parameter
$r$-oscillation seminorm is defined by
\begin{align}
\label{eq:102}
O_{I, J}^r(\mathfrak a_{t}(x): t \in \JJ):=
\Big(\sum_{j=0}^{J-1}\sup_{t\in \BB[I_j]\cap\JJ}
\abs{\mathfrak a_{t}(x) - \mathfrak a_{I_j}(x)}^r\Big)^{1/r},
\end{align}
where
$\BB[I_i]:=[I_{i1}, I_{(i+1)1})\times\ldots\times[I_{ik}, I_{(i+1)k})$
is a box determined by the element $I_i=(I_{i1}, \ldots, I_{ik})$ of
the sequence $I\in \mathfrak S_J(\II)$.  In order to avoid problems
with measurability we always assume that
$\II\ni t\mapsto \mathfrak a_{t}(x)\in\CC$ is continuous for
$\mu$-almost every $x\in X$, or $\JJ$ is countable. We also use the
convention that the supremum taken over the empty set is zero.  
\begin{remark}
\label{rem:1}
Let $1\le r<\infty$.
Some remarks  are in order.
\begin{enumerate}[label*={\arabic*}.]
\item Clearly $O_{I, J}^r(\mathfrak a_{t}: t \in \JJ)$ defines a semi-norm.

\item  Let $\II\subseteq \RR^k$ be an index set such that $\#{\II}\ge2$, and let $\JJ_1, \JJ_2\subseteq \II$ be disjoint. Then for any family $(\mathfrak a_t:t\in\II)\subseteq \CC$,  any $J\in\ZZ_+$ and any $I\in\mathfrak S_J(\II)$ one has
\begin{align*}
O_{I, J}^r(\mathfrak a_{t}: t\in\JJ_1\cup\JJ_2)
\le O_{I, J}^r(\mathfrak a_{t}: t\in\JJ_1)
+O_{I, J}^r(\mathfrak a_{t}: t\in\JJ_2).
\end{align*}

\item Let $\II\subseteq \RR^k$ be a countable index set such that $\#{\II}\ge2$ and $\JJ\subseteq \II$. Then 
 for any family $(\mathfrak a_t:t\in\II)\subseteq \CC$, any  $J\in\ZZ_+$, any  $I\in\mathfrak S_J(\II)$  one has
 \begin{align}
\label{eq:4}
 O_{I, J}^r(\mathfrak a_{t}: t \in \JJ)\lesssim \Big(\sum_{t\in\II}|\mathfrak a_{t}|^r\Big)^{1/r}.
 \end{align}

\item Let $(\mathfrak a_t:t\in\II^k)$ be a
$k$-parameter family of measurable functions on $X$. For any $\II\subseteq\RR$
with $\#{\II}\ge2$ and any sequence
$I=(I_i : i\in\NN_{\le J}) \in \mathfrak S_J(\II)$ of length
$J\in\ZZ_+\cup\{\infty\}$ we define the diagonal sequence
$\bar{I}=(\bar{I}_i : i\in\NN_{\le J}) \in \mathfrak S_J(\II^k)$ by
setting $\bar{I}_i=(I_i,\ldots,I_i)\in\II^k$ for each
$i\in\NN_{\le J}$. Then for any $p\in[1, \infty]$  and for any $\JJ\subseteq \II^k$ one has
\begin{align*}
\sup_{I\in \mathfrak S_J(\II)}\norm[\big]{O_{\bar{I}, J}^r(\mathfrak a_{t}: t \in \JJ)}_{L^p(X)}
\le
\sup_{I\in \mathfrak S_J(\II^k)}\norm[\big]{O_{I, J}^r(\mathfrak a_{t}: t \in \JJ)}_{L^p(X)}.
\end{align*}

\end{enumerate}
\end{remark}

We now show that oscillation semi-norms
always dominate maximal functions. 
\begin{proposition}
\label{prop:5}
Assume that $k\in\ZZ_+$, $\II\subseteq \RR$ be such that $\#{\II}\ge2$ and let $(\mathfrak a_{t}: t\in\II^k)$
be a $k$-parameter family of measurable functions on $X$. Then
for every $p\in[1, \infty]$ and $r\in[1, \infty)$ we have
\begin{align}
\label{eq:9}
\norm[\big]{\sup_{t \in (\II\setminus\{\sup\II\})^k}\abs{\mathfrak a_{t}}}_{L^p(X)}
\le \sup_{t \in \II^k}\norm{\mathfrak a_{t}}_{L^p(X)}
+\sup_{J\in\ZZ_+}\sup_{I\in \mathfrak S_J(\II)}\norm[\big]{O_{\bar{I}, J}^r(\mathfrak a_{t}: t \in \II^k)}_{L^p(X)},
\end{align}
where $\bar{I}\in\mathfrak S_J(\II^k)$ is the diagonal sequence
corresponding to a sequence $I\in \mathfrak S_J(\II)$ as in Remark~\ref{rem:1}.
\end{proposition}
\begin{proof}
Let $a=\inf\II$ and $b=\sup\II$. We see that $a<b$, since
$\#{\II}\ge2$. We choose a decreasing sequence
$(a_n:n\in\NN)\subseteq\II$ and an increasing sequence
$(b_n:n\in\NN)\subseteq\II$ such that $a\le a_n\le b_n\le b$ for every
$n\in\NN$ satisfying
\[
\lim_{n\to\infty}a_n=a,
\qquad \text{and}\qquad
\lim_{n\to\infty}b_n=b
\]
and such that $a_n = a$ for all $n \in \NN$ if $a \in \II$.
By the monotone convergence theorem we get
\begin{align*}
\norm[\big]{\sup_{t \in (\II\setminus\{\sup\II\})^k}\abs{\mathfrak a_{t}}}_{L^p(X)}
&=\lim_{n\to\infty}\norm[\big]{\sup_{t \in [a_n, b_n)^k\cap\II^k}\abs{\mathfrak a_{t}}}_{L^p(X)}\\
&\le\sup_{n} \ \norm{\mathfrak a_{\bar{a}_n}}_{L^p(X)}+
\sup_{n} \ \norm[\big]{\sup_{t \in [a_n, b_n)^k\cap\II^k}\abs{\mathfrak a_{t}-\mathfrak a_{\bar{a}_n}}}_{L^p(X)},
\end{align*}
where $\bar{a}_n=(a_n,\ldots, a_n)\in [a_n, b_n)^k \cap \II^k$,
and consequently we obtain \eqref{eq:9}.
\end{proof}

A remarkable feature of the oscillation seminorms is that they imply
pointwise convergence.  This property is formulated precisely in the
following proposition.

\begin{proposition}
\label{prop:4}
Let $(X, \calB(X), \mu)$ be a $\sigma$-finite measure space. For
$k\in\ZZ_+$ let $(\mathfrak a_{t}: t\in\RR_+^k)$ be a
$k$-parameter family of measurable functions on $X$. Suppose that
there are  $p, r\in[1, \infty)$  such that for any $J\in\ZZ_+$ one has
\begin{align*}
\sup_{I\in \mathfrak S_J(\RR_+)}
\norm[\big]{O_{\bar{I}, J}^r(\mathfrak a_{t}: t \in \RR_+^k)}_{L^p(X)}\le C_{p, r}(J),
\end{align*}
where
\begin{align*}
\lim_{J\to \infty}J^{-\frac{1}{p\vee r}}C_{p, r}(J)=0,
\end{align*}
and 
$\bar{I}\in\mathfrak S_J(\RR_+^k)$ is the diagonal sequence
corresponding to a sequence $I\in \mathfrak S_J(\RR_+)$ as in Remark~\ref{rem:1}.  Then the limits
\begin{align}
\label{eq:161}
\lim_{\min\{t_1,\ldots, t_k\}\to\infty}\mathfrak a_{(t_1,\ldots, t_k)},
\qquad \text{ and } \qquad
\lim_{\max\{t_1,\ldots, t_k\}\to 0}\mathfrak a_{(t_1,\ldots, t_k)},
\end{align}
exist  $\mu$-almost everywhere on $X$.
\end{proposition}

\begin{proof}
We only prove the first conclusion of \eqref{eq:161} as the second one
can be proved in much the same way.  Suppose by contradiction that
the first limit in \eqref{eq:161} does not exist $\mu$ almost everywhere on $X$. Since $\mu$ is a
$\sigma$-finite measure then there exists $X_0\subseteq X$ such that
$\mu(X_0)<\infty$, and also there is a small $\delta>0$ such that
\begin{align*}
\mu\big(\Set{x\in X_0}{\lim_{N\to\infty}\sup_{s, t\succeq \bar{N}}
\abs{\mathfrak a_{s}(x)-\mathfrak a_{t}(x)}>2\delta}\big)>2\delta,
\end{align*}
where $\bar{N}=(N,\ldots, N)\in\ZZ_+^k$.
For $N\in\ZZ_+$ define
\begin{align*}
A_N:=\Set{x\in X_0}{\sup_{s, t\succeq \bar{N}}\abs{\mathfrak a_{s}(x)-\mathfrak a_{t}(x)}>2\delta}.
\end{align*}
Note that 
$A_{N+1}\subseteq A_N$ for every $N\in\ZZ_+$, and consequently from the continuity of measure one has
\begin{align*}
\lim_{N\to\infty}\mu\big(\Set{x\in X_0}{\sup_{s, t\succeq \bar{N}}\abs{\mathfrak a_{s}(x)-\mathfrak a_{t}(x)}>2\delta}\big)>2\delta.
\end{align*}
Hence there is an $N_0\in\ZZ_+$ such that for every $N\ge N_0$, we have
\begin{align*}
\mu\big(\Set{x\in X_0}{\sup_{t\succeq \bar{N}}\abs{\mathfrak a_{t}(x)-\mathfrak a_{\bar{N}}(x)}>\delta}\big)>\delta.
\end{align*}
For $M, N\in\ZZ_+$ we now define
\begin{align*}
B_M^N:=\Set{x\in X_0}{\sup_{\bar{N}\preceq t\prec_{\rm s}\bar{M}}\abs{\mathfrak a_{t}(x)-\mathfrak a_{\bar{N}}(x)}>\delta}.
\end{align*}
We observe that $B_M^N\subseteq B_{M+1}^N$ for every $M, N\in\ZZ_+$ and
using once again continuity of measure, we obtain for every $N\ge N_0$,
\begin{align}
\label{eq:162}
\lim_{M\to\infty}\mu(B_M^N)=
\mu\big(\Set{x\in X_0}{\sup_{t\succeq \bar{N}}\abs{\mathfrak a_{t}(x)-\mathfrak a_{\bar{N}}(x)}>\delta}\big)>\delta.
\end{align}
Using \eqref{eq:162} recursively we can construct a strictly
increasing sequence $(I_i:i\in\NN)\subset \RR_+$ with $I_0=N_0$ such that for every
$i\in\NN$ we have
\begin{align}
\label{eq:163}
\mu\big(\Set{x\in X_0}{\sup_{t\in\BB[\bar{I}_i]}
\abs{\mathfrak a_{t}(x)-\mathfrak a_{\bar{I}_i}(x)}>\delta}\big)>\delta,
\end{align}
where $\bar{I}_i=(I_i,\ldots,I_i)\in\RR_+^k$.
Then by  \eqref{eq:163} we obtain for every $J\in\ZZ_+$ that
\begin{align*}
J\delta^{p+1}=\sum_{j=0}^{J-1}\delta^{p+1}&\le
\int_{X}\sum_{j=0}^{J-1}\sup_{t\in\BB[\bar{I}_j]}
\abs{\mathfrak a_{t}(x)-\mathfrak a_{\bar{I}_j}(x)}^p\dif\mu(x)\\
&\le J^{1-q/r}\sup_{I\in \mathfrak S_J(\ZZ_+)}
\norm[\big]{O_{\bar{I}, J}^r(\mathfrak a_{t}: t \in \RR_+^k)}_{L^p(X)}^p,
\end{align*}
where $q:=p\wedge r$.
Thus
\begin{align*}
J^{q/r}\delta^{p+1}
\le \sup_{I\in \mathfrak S_J(\RR_+)}
\norm[\big]{O_{\bar{I}, J}^r(\mathfrak a_{t}: t \in \RR_+^k)}_{L^p(X)}^p
\le C_{p, r}(J)^p.
\end{align*}
Letting $J\to \infty$ we get a contradiction.  This completes the
proof of Proposition \ref{prop:4}.
\end{proof}

\subsection{Variation semi-norms}
 We recall the definition of $r$-variations. For any $\mathbb I\subseteq \RR$, any family $(\mathfrak a_t: t\in\mathbb I)\subseteq \CC$, and any exponent
$1 \leq r < \infty$, the $r$-variation semi-norm is defined to be
\begin{align}
\label{eq:67}
V^{r}(\mathfrak a_t: t\in\mathbb I):=
\sup_{J\in\ZZ_+} \sup_{\substack{t_{0}<\dotsb<t_{J}\\ t_{j}\in\mathbb I}}
\Big(\sum_{j=0}^{J-1}  |\mathfrak a_{t_{j+1}}-\mathfrak a_{t_{j}}|^{r} \Big)^{1/r},
\end{align}
where the latter supremum is taken over all finite increasing sequences in
$\mathbb I$.

\begin{remark}
\label{rem:2}
Some remarks about definition \eqref{eq:67}  are in order.
\begin{enumerate}[label*={\arabic*}.]
\item Clearly $V^{r}(\mathfrak a_t: t\in\mathbb I)$ defines a semi-norm.
\item The function $r\mapsto V^r(\mathfrak{a}_t: t\in \II)$ is non-increasing. Moreover, if $\II_1\subseteq \II_2$, then
\begin{equation*}
    V^r(\mathfrak{a}_t: t\in \II_1)\le V^r(\mathfrak{a}_t: t\in \II_2).
\end{equation*}

 \item  Let $\II\subseteq \RR$ be  such that $\#{\II}\ge2$. Let  $(\mathfrak a_t:t\in\RR)\subseteq \CC$ be given, and let $r\in[1, \infty)$. If $V^r(\mathfrak{a}_t: t\in \RR)<\infty$ then
$\lim_{t\to\infty}\mathfrak{a}_t$ exists. Moreover, for any $t_0\in\II$ one has
\begin{align}
\label{eq:68}
\sup_{t\in\II}|\mathfrak{a}_t|\le |\mathfrak{a}_{t_0}|+V^r(\mathfrak{a}_t: t\in \II).
\end{align}

\item Let $\II\subseteq \RR$ be 
such that $\#{\II}\ge2$. Then 
 for any $r\ge1$, and any family $(\mathfrak a_t:t\in\II)\subseteq \CC$, any  $J\in\ZZ_+\cup\{\infty\}$, any  $I\in\mathfrak S_J(\II)$  one has
 \begin{align}
\label{eq:62}
 O_{I, J}^r(\mathfrak a_{t}: t \in \II)\le V^r(\mathfrak{a}_t: t\in \II)\le 2\Big(\sum_{t\in\II}|\mathfrak a_{t}|^r\Big)^{1/r}.
 \end{align}

\item Let $(\mathfrak{a}_t(x): t\in\RR_+)$ be a family of complex-valued measurable functions on a $\sigma$-finite measure space $(X,\calB(X),\mu)$. Then for any $p\ge1$ and $r \ge 2$ we have 
\begin{align}
\label{eq:70}
\begin{gathered}
\qquad \sup_{N\in\ZZ_+}\sup_{I\in\mathfrak{S}_N(\RR_+)}\norm{O_{I,N}^r(\mathfrak{a}_t:t\in\RR_+)}_{L^p(X)}\lesssim \sup_{N\in\ZZ_+}\sup_{I\in\mathfrak{S}_N(\DD)}\norm{O_{I,N}^r(\mathfrak{a}_{t}:t\in \DD)}_{L^p(X)}\\
\qquad +\norm[\Big]{\Big(\sum_{n\in\ZZ} V^r(\mathfrak{a}_{t}:t\in[2^{n},2^{n+1}])^2\Big)^{1/2}}_{L^p(X)}.
\end{gathered}
\end{align}
The inequality \eqref{eq:70} is an analogue of \cite[Lemma 1.3, p. 6716]{JSW} for oscillation semi-norms. 
\end{enumerate}

\end{remark}

\subsection{Jumps}
The $r$-variation is closely related to the $\lambda$-jump counting function. Recall that for any $\lambda>0$ the $\lambda$-jump counting function of a function $f : \I \to \C$ is defined by
\begin{align}
\label{eq:71}
N_{\lambda} f:=N_\lambda(f(t):t\in\II):=\sup \{ J\in\N : \exists_{\substack{t_{0}<\ldots<t_{J}\\ t_{j}\in\I}}  : \min_{0 \le j \le J-1} |f(t_{j+1})-f(t_{j})| \ge \lambda \}.
\end{align}

\begin{remark}
\label{rem:3}
Some remarks about definition \eqref{eq:71}  are in order.
\begin{enumerate}[label*={\arabic*}.]
\item For any $\lambda>0$ and a function $f : \I \to \C$ let us also define the following quantity
\begin{align*}
\qquad \mathcal{N}_{\lambda} f:=\calN_\lambda(f(t):t\in\II):=\sup \{ J\in\N : \exists_{\substack{s_1 < t_{1}  \le \ldots \le s_J < t_{J} \\ s_j , t_{j}\in\I}}  : \min_{1 \le j \le J} |f(t_{j})-f(s_{j})| \ge \lambda \}.  
\end{align*}
Then on has $N_{\lambda} f \le \mathcal{N}_{\lambda} f \le N_{\lambda/2} f$.

\item It is clear from these definitions that 
$f\mapsto \sup_{\lambda>0} \big\| \lambda N_{\lambda}(f(\cdot, t): t\in \II)^{1/\rho} \big\|_{L^{p}(X)}$
satisfies a quasi-triangle inequality. However it is not obvious whether  a genuine triangle inequality is available for
$\lambda$-jumps. 
In many applications, the problem can be overcome since there is always a comparable semi-norm in the following sense.
Namely, for every
$p \in (1,\infty)$, and $\rho \in (1,\infty)$ there exists a constant
$0<C<\infty$ such that for every measure space $(X,\calB(X),\mu)$, and
$\II\subseteq \RR$, there exists a (subadditive) seminorm
$\vertiii{\cdot}$ such that the following two-sided inequality
\begin{align*}
C^{-1} \vertiii{f} \leq \sup_{\lambda>0} \big\| \lambda N_{\lambda}(f(\cdot, t): t\in \II)^{1/\rho} \big\|_{L^{p}(X)} \leq C \vertiii{f}
\end{align*}
 holds for all measurable functions $f : X \times \I \to \CC$. This was established in
\cite[Corollary 2.2, p. 805]{MSZ1}.

\item  Let $(\mathfrak{a}_t(x): t\in\RR)$ be a family of measurable functions on a $\sigma$-finite measure space $(X,\calB(X),\mu)$. Let $\II\subseteq\RR$ and $\#\II\geq2$, then for every $p\in[1,\infty]$ and $r\in[1,\infty)$ we have 
\begin{align}
\label{eq:72}
\sup_{\lambda>0}\|\lambda N_{\lambda}(\mathfrak{a}_t: t\in\II)^{1/r}\|_{L^p(X)}\le \|V^r(\mathfrak{a}_t: t\in\II)\|_{L^p(X)},
\end{align}
since for all $\lambda>0$ we have the following pointwise estimate
\[
\lambda N_{\lambda}(\mathfrak{a}_t(x): t\in\II)^{1/r}\le V^r(\mathfrak{a}_t(x): t\in\II).
\]
\item  Let $(X,\mathcal{B}(X),\mu)$ be a $\sigma$-finite measure space and $\I\subseteq\RR$. Fix $p \in [1,\8]$, and $1 \le \rho < r \le \infty$. Then for every measurable function $f : X \times \I \to \C$ we have the estimate
\begin{align} \label{estt1}
\big\| V^{r}\big( f(\cdot, t) : t \in \I \big) \big\|_{L^{p,\8} (X)}
\lesssim_{p,\rho,r}
\sup_{\lambda>0} \big\| \lambda N_{\lambda}(f(\cdot, t): t\in \I)^{1/\rho} \big\|_{L^{p,\8}(X)}.
\end{align}
The inequality \eqref{estt1} can be thought of as an inverse to  inequality \eqref{eq:72}. A proof of \eqref{estt1} can be found in \cite[Lemma 2.3, p. 805]{MSZ1}. Moreover, one cannot replace $L^{p,\8} (X)$  with $L^{p} (X)$ in \eqref{estt1}, see \cite[Lemma 2.24]{MSS}. One can also show that there is  $f:\ZZ_+\times\ZZ_+\to\RR$ such that  
\begin{equation*}
\sup_{N\in\ZZ_+}\sup_{I\in\mathfrak{S}_N(\ZZ_+)}\norm{O_{I,N}^r(f(\cdot,n):n\in\ZZ_+)}_{\ell^p(\ZZ_+)}
=\infty,\quad 2\le r\le\infty,
\end{equation*}
but
\begin{equation*}
\sup_{\lambda > 0} \| \lambda N_{\lambda}(f(\cdot,n):n\in\ZZ_+)^{1/2} \|_{\ell^p(\ZZ_+)} < \infty.
\end{equation*}

\end{enumerate}
\end{remark}

\section{One-parameter oscillation estimates}
\label{section:3}

We state a simple one-parameter oscillation estimate for projections, which has many interesting implications. Here we are inspired by observations of M. Lacey who highlighted and pointed out the importance of projections in pointwise ergodic theory; see \cite{RW}.

\begin{theorem} \label{thm:p}
Let $(X,\mathcal B(X), \mu)$ be a $\sigma$-finite measure space and let  $\II\subseteq \RR$ be such that $\#\II\ge2$. Let $(P_t)_{t \in \II}$ be a family of projections; that is, the linear operators $P_t:L^0(X)\to L^0(X)$ satisfying
\begin{align}
\label{eq:10}
P_s P_t = P_{s \wedge t}, \qquad \text{ for} \qquad  s \not=  t. 
\end{align}
If the set $\II$ is uncountable then we assume in addition that $\II\ni t\mapsto P_tf$ is
continuous $\mu$-almost everywhere on $X$ for every $f\in L^0(X)$.
Let $p, r \in(1, \infty)$ be fixed. Suppose that
$P_t$ are bounded on $L^p(X)$, and suppose that the
following two estimates hold
\begin{align} \label{1.15}
\sup_{J\in\ZZ_+}\sup_{I\in \mathfrak{S}_J(\II)} 
\norm[\Big]{ \Big( \sum_{j=0}^{J-1} 
\abs{ (P_{I_{j+1}}- P_{I_{j}}) f}^r \Big)^{1/r} }_{L^p(X)}
\lesssim_{p, r} 
\norm{f}_{L^p(X)}, \qquad f \in L^p(X),
\end{align}
and the vector-valued estimate
\begin{align} \label{1.2}
\norm[\Big]{ \Big( \sum_{j \in \ZZ} \sup_{t \in \II} \abs{P_{t} f_j}^r \Big)^{1/r} }_{L^p(X)} 
\lesssim_{p, r} 
\norm[\Big]{ \Big( \sum_{j \in \ZZ} \abs{f_j}^r \Big)^{1/r} }_{L^p(X)}, \qquad (f_j)_{j\in\ZZ}\in L^p(X; \ell^r(\ZZ)).
\end{align}
Then the following one-parameter oscillation estimate holds: 
\begin{align} \label{1.3}
\sup_{J \in \ZZ_+} \sup_{ I \in \mathfrak S_{J}(\II) }
\norm{ O^r_{I, J} ( P_{t} f : t \in \II) }_{L^p(X)} 
\lesssim_{p, r}
\norm{ f }_{L^p(X)}, 
\qquad f \in L^p(X).
\end{align}
\end{theorem}

\begin{proof} 
Fix $J \in \ZZ_+$ and $I \in \mathfrak S_{J}(\II)$ and
observe, using \eqref{eq:10},  that
\begin{align*} 
(P_t - P_{I_{j}} )f = P_t (P_{I_{j+1}} - P_{I_{j}}) f, 
\qquad \text{ whenever } \qquad I_j < t < I_{j+1}.
\end{align*}
Using this identity and then \eqref{1.2} we see that 
\begin{align*} 
\norm[\Big]{ \Big( \sum_{j=0}^{J-1} \sup_{\substack{I_j < t<I_{j+1}\\ t\in \II}} \abs{P_t f - P_{I_{j}} f}^r \Big)^{1/r} }_{L^p(X)}
& \le 
\norm[\Big]{ \Big( \sum_{j=0}^{J-1} \sup_{t \in \II} 
\abs{P_t (P_{I_{j+1}} - P_{I_{j}}) f}^r \Big)^{1/r} }_{L^p(X)} \\
& \lesssim_{p, r}
\norm[\Big]{ \Big( \sum_{j=0}^{J-1} 
\abs{ (P_{I_{j+1}} - P_{I_{j}}) f}^r \Big)^{1/r} }_{L^p(X)}.
\end{align*}
Now applying \eqref{1.15} we arrive at \eqref{1.3}.
The proof of Theorem~\ref{thm:p} is complete.
\end{proof} 

\begin{remark} \label{rem:33}
A few remarks are in order.
\begin{enumerate}[label*={\arabic*}.]
\item Theorem~\ref{thm:p} will be applied mainly when $r=2$. Then the estimate in \eqref{1.15} is a square function estimate, which can be deduced from the estimate
\begin{align}
\label{1.1}
\sup_{J \in \ZZ_+} \sup_{ I \in \mathfrak S_{J}(\II) } \sup_{\substack{\abs{\varepsilon_j} \le 1\\ 0 \le j\le J}}
\norm[\Big]{\sum_{j = 0}^{J-1} \varepsilon_j (P_{I_{j+1}} f - P_{I_{j}} f) }_{L^p(X)} 
\lesssim_p 
\norm{f}_{L^p(X)}, \qquad f \in L^p(X).
\end{align}
In fact, the implication from \eqref{1.1} to \eqref{1.15} is a simple consequence of Khintchine's inequality.

\item Let $(X,\mathcal B(X), \mu)$ be a $\sigma$-finite measure space, $\II\subseteq \RR$ be countable and let $(T_t)_{t \in \II}$ be a family of bounded operators on $L^p(X)$ for $p\in(1, \infty)$ satisfying
\begin{align}\label{eq:12}
\norm[\Big]{ \Big(\sum_{t \in \II} \abs{(T_t-P_{t})f}^2 \Big)^{1/2} }_{L^p(X)} 
\lesssim_p 
\norm{f}_{L^p(X)}, \qquad f \in L^p(X),
\end{align}
where $(P_t)_{t \in \II}$ is a family of projections as in Theorem~\ref{thm:p} satisfying \eqref{1.15} and \eqref{1.2} with $r=2$. Then one has
\begin{align}
\label{eq:13}
\sup_{J \in \ZZ_+} \sup_{ I \in \mathfrak S_{J}(\II) }
\norm{ O^2_{I, J} ( T_{t} f :t \in \II) }_{L^p(X)} 
\lesssim_{p}
\norm{ f }_{L^p(X)}, 
\qquad f \in L^p(X).
\end{align}
In fact, in view of \eqref{eq:62} the inequality \eqref{eq:12} easily reduces the 2-oscillation estimate for $(T_t)_{t \in \II}$ to a 2-oscillation estimate for $(P_t)_{t \in \II}$.
This observation will be very useful in many applications. We will see how it works in the case of 
smooth bump functions, see Theorem~\ref{thm:1}.

\item As we know oscillation inequalities are important in pointwise convergence  problems, and in the vast majority of applications it suffices to understand \eqref{eq:13} for $p=2$.  This can be nicely illustrated as follows: suppose for $p\in(1, \infty)$ one has an a priori maximal bound
\begin{align}
\label{eq:14}
\norm[big]{\sup_{t \in \II} \abs{P_{t} f}}_{L^p(X)}\lesssim_p\norm{f}_{L^p(X)}, \qquad f \in L^p(X).
\end{align}
Then \eqref{eq:14} with $p=2$  can be used to verify \eqref{1.2} with $p=r=2$. Finally, it remains to  verify \eqref{1.15} with $p=r=2$, which in many cases can be deduced by using Fourier techniques or exploiting almost-orthogonality phenomena invoking $TT^*$ arguments, see Proposition \ref{prop:6}. 
\end{enumerate}
\end{remark}

We now derive some consequences of Theorem~\ref{thm:p}.

\subsection{Oscillation inequalities for martingales}

We recall some basic facts about martingales. We will follow  notation from \cite[Section 3, p. 165]{HNVW}. Let
$(X,\calB(X),\mu)$ be a $\sigma$-finite measure space and let $\II$ be
a totally ordered set. A sequence of sub-$\sigma$-algebras
$(\mathcal F_t: t\in\II)$ is called a \textit{filtration} if it is increasing and
the measure $\mu$ is $\sigma$-finite on each $\mathcal F_t$.
A \textit{martingale} adapted to a filtration $(\mathcal F_t: t\in\II)$ is a family of functions $\mathfrak f=(\mathfrak f_t:t\in\II)\subseteq L^1(X,\calB(X),\mu)$ such that $\mathfrak f_{s}=\EE[\mathfrak f_t|\mathcal F_s]$ for every $s, t\in\II$ so that $s\le t$, where $\EE[\cdot|\mathcal F]$ denotes the the conditional
expectation operator with respect to a sub-$\sigma$-algebra $\mathcal F\subseteq \mathcal B(X)$. We say that a martingale 
$\mathfrak f=(\mathfrak f_t:t\in\II)\subseteq L^p(X,\calB(X),\mu)$ is bounded if
\[
\sup_{t\in\II}\|\mathfrak f_t\|_{L^p(X)}\lesssim_p1.
\]

Applying Theorem~\ref{thm:p} we immediately recover the oscillation inequality of Jones--Kaufman--Rosenblatt--Wierdl \cite{jkrw}, which in fact is an oscillation inequality for bounded martingales.

\begin{proposition}\label{prop:1}
For every $p\in(1, \infty)$ there exists a constant $C_p>0$ such that for every
bounded martingale $\mathfrak f=(\mathfrak f_n:n\in\ZZ)\subseteq L^p(X,\calB(X),\mu)$ corresponding to a filtration $(\mathcal F_n: n\in\ZZ)$ one has
\begin{align}
\label{eq:11}
\sup_{J\in\ZZ_+}\sup_{I\in\mathfrak S_J(\ZZ)}
\norm[\big]{O^2_{I,J}(\mathfrak f_n:n\in\ZZ)}_{L^p(X)}
&\leq C_p
\sup_{n\in\ZZ}\|\mathfrak f_n\|_{L^p(X)}.
\end{align}
\end{proposition}

Inequality \eqref{eq:11} was established in  \cite[Theorem 6.4, p. 930]{jkrw}. The authors first established \eqref{eq:11} for $p=2$, then proved weak type $(1,1)$ as well as $L^{\infty}\to {\rm BMO}$ variants of \eqref{eq:11}, and consequently derived \eqref{eq:11} for all $p\in(1, \infty)$ by interpolation.
Our approach is direct and will avoid using any interpolation arguments in the proof. 

\begin{proof}[Proof of Proposition \ref{prop:1}]
Fix $p\in(1, \infty)$. Define  projections by $P_n(f):=\EE[f|\calF_n]$ for any $n\in\ZZ$ and $f\in L^p(X)$. Since $\mathfrak f=(\mathfrak f_n:n\in\ZZ)$ is a martingale then $\mathfrak f_n=P_n(\mathfrak f_n)$ for any $n\in\ZZ$ and consequently \eqref{eq:10} holds. Moreover, by Burkholder \cite{Bu1}, see also \cite{Bu2},  it is very well known that \eqref{1.1} holds, which in view of Remark \ref{rem:33}, implies
\begin{align*}
\sup_{J\in\ZZ_+}\sup_{I\in \mathfrak{S}_J(\ZZ_+)} 
\norm[\Big]{ \Big( \sum_{j=0}^{J-1} 
\abs{ P_{I_{j+1}}(\mathfrak f_{I_J})- P_{I_{j}}(\mathfrak f_{I_J})}^2 \Big)^{1/2} }_{L^p(X)}
\lesssim_{p} 
\sup_{n\in\ZZ}\|\mathfrak f_n\|_{L^p(X)}.
\end{align*}
This consequently verifies inequality \eqref{1.15}. Invoking the Fefferman--Stein inequality for non-negative submartingales \cite[Theorem 3.2.7, p. 178]{HNVW} we obtain
\begin{align*}
\norm[\Big]{ \Big( \sum_{j \in \ZZ} \sup_{n \in \ZZ} \abs{\EE[|f_j||\calF_n]}^2 \Big)^{1/2} }_{L^p(X)} 
\lesssim_{p} 
\norm[\Big]{ \Big( \sum_{j \in \ZZ} \abs{f_j}^2 \Big)^{1/2} }_{L^p(X)}, \qquad (f_j)_{j\in\ZZ}\in L^p(X; \ell^2(\ZZ)),
\end{align*}
which in turn verifies the vector-valued estimate from \eqref{1.2}. Appealing to Theorem \ref{thm:p}, the oscillation inequality \eqref{eq:11} follows and the proof of Proposition \ref{prop:1} is complete. 
\end{proof}

\subsection{Oscillation inequalities for smooth bump functions}
Our aim will be to show that oscillation inequalities hold for $L^1$-dilated smooth bump functions. We begin with the main estimate.
\begin{proposition}
\label{prop:2}
For $d\in\ZZ_+$ let $\chi:\RR^d\to[0, 1]$ be a smooth function satisfying
\begin{align}
\label{eq:15}
\ind{[-1,1]^d} \le \chi \le \ind{[-2,2]^d}
\quad \text{ for } \quad
\xi\in\RR^d.
\end{align}
For every $n\in \ZZ$ and $\xi\in\RR^d$ define
$\chi_{2^n}(\xi):=\chi(2^{-n}\xi)$. Then for every $p\in(1, \infty)$ one has
\begin{align}
\label{eq:16}
\sup_{J \in \ZZ_+} \sup_{ I \in \mathfrak S_{J}(\ZZ) }
\norm{ O^2_{I, J} ( T_{\RR^d}[\chi_{2^n}]f : n \in \ZZ) }_{L^p(\RR^d)} 
\lesssim_{p}
\norm{ f }_{L^p(\RR^d)}, 
\qquad f \in L^p(\RR^d).
\end{align}
\end{proposition}

\begin{proof}
Setting $P_nf:= T_{\RR^d}[\chi_{2^{n}}]f$ for every $n\in\ZZ$, and using \eqref{eq:15} one sees that $P_n$ is  a projection in the sense of \eqref{eq:10}. Standard arguments based on the Littlewood--Paley theory 
(see Section \ref{lp}) show that \eqref{1.15} with $r=2$ holds. By the Fefferman--Stein inequality \cite{bigs} we also obtain \eqref{1.2}. An application of Theorem \ref{thm:p} now gives \eqref{eq:16} as desired.
\end{proof}

Now our aim will be to extend inequality \eqref{eq:16} to continuous times and general smooth bump functions. 

\begin{remark} \label{rem:100}
A few remarks concerning Proposition \ref{prop:2} are in order. 

\begin{enumerate}[label*={\arabic*}.]
\item An important feature of our approach in Proposition \ref{prop:2}
is that we do not need to invoke the corresponding inequality for
martingales in the proof. This stands in sharp contrast to variants
of inequality \eqref{eq:16} involving $r$-variations, where all
arguments to the best of our knowledge use the corresponding
$r$-variational inequalities for martingales.

\item Of course, inequality \eqref{eq:16} can be reduced to the
martingale setting from Proposition \ref{prop:1} by invoking square
function arguments \cite[Lemma 3.2, p. 6722]{JSW} and standard
Littlewood--Paley theory. The details may be found in \cite{MSZ2}.

\item With respect to the previous two remarks, it would be interesting to know whether the $r$-variational counterpart of Proposition \ref{prop:2} can be proved without appealing to $r$-variational inequalities for martingales, see L{\'e}pingle's inequality \eqref{eq:42}. 
\end{enumerate}
\end{remark}

\begin{theorem}
\label{thm:1}
For $d\in\ZZ_+$ let $\phi : \RR^d\to \CC$ be a Schwartz function. For $t\in\RR_+$ and $x\in\RR^d$ define $\phi_t (x) := t^{-d} \phi(t^{-1}x)$. Then for every $p\in (1, \infty)$ one has
\begin{align}
\label{id:1}
\sup_{J \in \ZZ_+} \sup_{ I \in \mathfrak S_{J}(\RR_+) }
\norm{ O^2_{I, J} ( \phi_t * f : t \in\RR_+) }_{L^p(\RR^d)} 
\lesssim_p 
\norm{ f }_{L^p(\RR^d)}, 
\qquad f \in L^p(\RR^d).
\end{align}
\end{theorem}

\begin{remark}\label{rem:partial} Theorem \ref{thm:1} immediately extends to families
of partial convolution operators. If $\RR^d = \RR^n \times \RR^m$, we write elements $x\in \RR^d$ as $x = (x',x'')$ where 
$x'\in \RR^n$ and $x''\in \RR^m$. Let $\phi$ be a Schwartz function on $\RR^n$ and define
$$
T_t f(x) \ = \ \int_{\RR^n} f(x' - y,x'') \phi_t(y) \, dy.
$$
The oscillation inequality \eqref{id:1} implies the corresponding oscillation inequality for the family 
of partial convolution operators $(T_t)_{t \in \RR_+}$.
\end{remark}

\begin{proof}[Proof of Theorem \ref{thm:1}]
To prove \eqref{id:1}, in view of \eqref{eq:70}, it suffices to show
\begin{align}
\label{4.2}
\sup_{J \in \ZZ_+} \sup_{ I \in \mathfrak S_{J}(\DD) }
\norm{ O^2_{I, J} ( \phi_t * f : t \in \DD) }_{L^p(\RR^d)} 
\lesssim_p 
\norm{ f }_{L^p(\RR^d)}, 
\qquad f \in L^p(\RR^d),
\end{align}
and
\begin{align} \label{7.1}
\norm[\Big]{ \Big( \sum_{k \in \ZZ} V^2 \big(\phi_t * f : t \in [2^k,2^{k+1}] \big)^2 \Big)^{1/2} }_{L^{p}(\RR^d)}
\lesssim
\norm{f }_{L^{p}(\RR^d)}, \qquad f \in L^{p}(\RR^d).
\end{align}

Short 2-variational estimates were treated in \cite{JSW} and in particular, the estimate \eqref{7.1} follows directly from 
\cite[Lemma 6.1]{JSW}. 
\vskip 5pt
To establish \eqref{4.2} we first observe that we may assume that $\int_{\RR^d} \phi(x)dx = 0$. Indeed, if $\int_{\RR^d} \phi(x)dx \ne 0$, then by scaling we may assume that $\int_{\RR^d} \phi(x)dx = \chi(0) = 1$ where $\chi$ appears in Proposition~\ref{prop:2}. By standard Littlewood--Paley arguments (see 
Section \ref{lp}), we note that \eqref{eq:12} holds with $T_t f = \phi_t * f$ and 
$P_t f = T_{\RR^d}[\chi_t] f$. Therefore by Remark \ref{rem:33}, we see that \eqref{4.2} follows from the 
oscillation inequality \eqref{eq:16} and so we may assume $\phi$ has mean zero.
Using \eqref{eq:4} we see that
\begin{align*} 
\text{ LHS of } \eqref{4.2} 
& \lesssim
\norm[\Big]{ \Big( \sum_{k \in \ZZ} \abs{\phi_{2^k} * f}^2 \Big)^{1/2} }_{L^p(\RR^d)} 
\lesssim \|f\|_{L^p(\RR^d)},
\end{align*}
the last inequality following directly from \cite[Theorem B]{DR}; see Section \ref{lp}.
This completes the proof of Theorem~\ref{thm:1}.
\end{proof}

\subsection{Oscillation inequalities for orthonormal systems}
We now state a result which justifies in a strong sense the importance of oscillation
inequalities.

\begin{proposition}
\label{prop:6}
Let $(X,\calB(X),\mu)$ be a $\sigma$-finite measure space such that
the corresponding Hilbert space $L^2(X)$ is endowed  with an orthonormal basis $(\Phi_n)_{n\in\NN}$. Then the
projection operators
\begin{align}
\label{eq:17}
P_nf:=\sum_{k=0}^n\langle f, \Phi_k\rangle \Phi_k, \qquad f\in L^2(X),
\end{align}
satisfy the oscillation estimate
\begin{align}
\label{eq:20-}
\sup_{J \in \ZZ_+} \sup_{ I \in \mathfrak S_{J}(\NN_{\le N}) }
\norm{ O^2_{I, J} ( P_{n} f : n \in \NN_{\le N}) }_{L^2(X)} 
\lesssim \ \log(N+1) \,
\norm{ f }_{L^2(X)}.
\end{align}
Furthermore if the projection operators $P_n$
satisfy the following maximal estimate
\begin{align}
\label{eq:18}
\norm[\big]{\sup_{n\in\NN}|P_nf|}_{L^2(X)}\lesssim \norm{f}_{L^2(X)} , \qquad f\in L^2(X),
\end{align}
then one has  the uniform bound
\begin{align}
\label{eq:20}
\sup_{J \in \ZZ_+} \sup_{ I \in \mathfrak S_{J}(\NN) }
\norm{ O^2_{I, J} ( P_{n} f : n \in \NN) }_{L^2(X)} 
\lesssim
\norm{ f }_{L^2(X)}, 
\qquad f \in L^2(X).
\end{align}
\end{proposition}

\begin{proof}
It is easy to see that $P_n$ from \eqref{eq:17} satisfies  \eqref{eq:10}. To verify \eqref{1.15} we fix $J \in \ZZ_+$ and $I \in \mathfrak S_{J}(\NN)$ and note that by orthognality we have 
\begin{align*}
\norm[\Big]{ \Big( \sum_{j=0}^{J-1} 
\abs{ (P_{I_{j+1}}- P_{I_{j}}) f}^2 \Big)^{1/2} }_{L^2(X)}^2&=
\sum_{j=0}^{J-1} \sum_{k_1=I_j+1}^{I_{j+1}}\sum_{k_2=I_j+1}^{I_{j+1}}
\langle f, \Phi_{k_1}\rangle\overline{\langle f, \Phi_{k_2}\rangle}
\langle\Phi_{k_1}, \Phi_{k_2}\rangle\\
&\le\sum_{k\in\NN}
|\langle f, \Phi_{k}\rangle|^2\\
&=\|f\|_{L^2(X)}^2,
\end{align*}
where in the last line  we have used Parseval's identity for orthonormal bases. This proves \eqref{1.15} with $p=r=2$. 
A famous result of Rademacher \cite{Rad} and Menshov
\cite{Men} asserts that there is a constant $C>0$ such that for
any $N\in\ZZ_+$ the
projection operator $P_n$ from \eqref{eq:17} satisfies
\begin{align}
\label{eq:19}
\norm[\big]{\sup_{n\in[N]}|P_nf|}_{L^2(X)}\le C\log(N+1)\Big(\sum_{n\in[N]}|\langle f, \Phi_n\rangle|^2\Big)^{1/2} \
\lesssim \ \log(N+1) \, \|f\|_{L^2(X)}.
\end{align}
Using \eqref{eq:19} we see that \eqref{1.2} holds with $p=r=2$ with constant $\log(N+1)$. Now applying Theorem \ref{thm:p} we obtain \eqref{eq:20-}.

Under condition \eqref{eq:18}, we see that
\eqref{1.2} holds with a uniform constant for $p=r=2$ and so applying Theorem \ref{thm:p} again, we obtain \eqref{eq:20}.
\end{proof}

Proposition \ref{prop:6} is a key example in the study of oscillation
semi-norms from the point of view their importance and usefulness in
pointwise convergence problems. It exhibits, in view of inequality
\eqref{eq:9}, that oscillation estimates \eqref{eq:20} and maximal
estimates \eqref{eq:18} are equivalent in the class of orthonormal
systems. 

However, we have to emphasize that
the maximal estimate from \eqref{eq:18} is a very strong condition.  On the
one hand, we have Menshov's construction \cite{Men} of an orthonormal basis
$(\Psi_n)_{n\in\NN}\subseteq L^2([0, 1])$ and a function
$f_0\in L^2([0, 1])$ with almost everywhere diverging partial sums
$\sum_{k=0}^n\langle f, \Psi_k\rangle \Psi_k$. Therefore maximal
estimate \eqref{eq:18} for Menshov's system cannot hold. In fact,
the best what we can expect in the general case is the Rademacher--Menshov bound \eqref{eq:19}.
The above-mentioned Menshov's \cite{Men} construction also shows that
\eqref{eq:19} is sharp and the logarithm in \eqref{eq:19} cannot be
removed. 

On the other hand, there is the famous result of
Carleson \cite{Carleson} which led to establishing
\eqref{eq:18} for the canonical trigonometric system
$(\ex(n\xi))_{n\in\ZZ}$ on $L^2([0, 1])$ (see also
\cite{Hunt, F1, LT1}). 

\subsection{Oscillation inequalities for the Carleson operator}
In this subsection we obtain certain $r$-oscillation estimates for partial Fourier integrals on the real line $\RR$.
\vskip 5pt
The Carleson operator $\calC_t$
is defined by 
\begin{align}
\label{eq:23}
\calC_t f(x) := T_{\RR}[\ind{[-t, t]}]f(x)= \int_{-t}^t \calF_{\RR}f(\xi) \ex(-x \xi) \, d\xi, 
\qquad f \in \calS(\RR), \quad x\in\RR,\quad  t \in \RR_+.
\end{align}
The celebrated  Carleson--Hunt theorem (see the papers of Carleson \cite{Carleson} and Hunt
\cite{Hunt}) asserts that for every $p\in(1, \infty)$ there is a
constant $C_p>0$ such that
\begin{align}
\label{eq:21}
\norm[\big]{\sup_{t>0}|\calC_t f|}_{L^p(\RR)}\le C_p\norm{f}_{L^p(\RR)}, \qquad f\in L^p(\RR).
\end{align}
\begin{remark} \label{rem:4}
A few remarks about the Carleson--Hunt theorem are in order. 

\begin{enumerate}[label*={\arabic*}.]

\item Carleson \cite{Carleson} originally proved that the maximal
partial sum operator of Fourier series corresponding to
square-integrable functions on the circle is weak type $(2, 2)$. Not
long afterwards this result was extended by Hunt \cite{Hunt} who
proved that the maximal partial sum operator of Fourier series is
bounded on $L^p(\TT)$ for any $p\in(1, \infty)$.  
\item Kenig and Tomas \cite{KT} used a transplantation
arugment to show the latter
result is equivalent to inequality \eqref{eq:21}. This equivalence was extended to variation and oscillation inequalities
in \cite{OSTTW}. The foundational  work of  Kolmogorov \cite{Kol1, Kol2} shows
that the range of $p\in(1, \infty)$ in inequality \eqref{eq:21} is sharp.
\item An alternative proof of Carleson's theorem was provided by
Fefferman \cite{F1}, who pioneered the ideas of the so called time--frequency analysis.
\item Lacey and Thiele \cite{LT1} established an independent
proof on the real line of the weak type $(2, 2)$ boundedness of the
maximal Fourier integral operator \eqref{eq:23}. The latter bound was
extended by Grafakos, Tao, and Terwilleger \cite{GTT} to \eqref{eq:21} for all $p\in(1, \infty)$, see also \cite{PT}.
\item Inequality \eqref{eq:21} was extended to vector-valued
setting by Grafakos, Martell and Soria \cite{GMS}, who proved that
that for every $p, r\in(1, \infty)$ there is a constant $C_{p, r}>0$ such that
\begin{align}
\label{eq:26}
\qquad \quad \norm[\Big]{ \Big( \sum_{j \in \ZZ} \sup_{t>0}|\calC_t f_j|^r \Big)^{1/r} }_{L^p(\RR)} 
\le C_{p, r}
\norm[\Big]{ \Big( \sum_{j \in \ZZ} \abs{f_j}^r \Big)^{1/r} }_{L^p(\RR)}, \qquad (f_j)_{j\in\ZZ}\in L^p(X; \ell^r(\ZZ)).
\end{align}

\item We finally refer to the survey of Lacey \cite{La1}, where
details (including comprehensive historical background) and an
extensive literature are given about this fascinating subject of
pointwise convergence of Fourier series and related topics.
\end{enumerate}
\end{remark}

A far-reaching quantitative extension of \eqref{eq:21} was obtained
by the third author in collaboration with Oberlin, Seeger, Tao and
Thiele \cite{OSTTW}, which asserts that for every $p\in(1, \infty)$
and for every $r>\max\big\{2, \frac{p}{p-1}\big\}$ there is a constant
$C_{p, r}>0$ such that
\begin{align}
\label{eq:24}
\norm[\big]{V^r(\calC_t f: t\in\RR_+)}_{L^p(\RR)}\le C_{p, r}\norm{f}_{L^p(\RR)}, \qquad f\in L^p(\RR).
\end{align}
See also in \cite{U} for a different proof using outer measures. 
Furthermore a restricted weak-type bound is established at the endpoint $p = r'$ when $p\in (1,2)$ (here $r' = r/(r-1)$) and it is open whether weak type $(p,p)$ holds true.
It also follows from \cite{OSTTW} that the ranges of parameter
$p\in(1, \infty)$ and $r>\max\big\{2,p'\big\}$ in
\eqref{eq:24} are sharp. In the endpoint case $p=r'$, the Lorentz space $L^{r',\infty}$ cannot be replaced by a 
smaller Lorentz space. For weighted variational estimates for the Carleson operator, see  \cite{DL}, and \cite{DDU} and the references given there.

Inequality \eqref{eq:24}, in view of inequality \eqref{eq:62},
immediately implies that for every $p\in(1, \infty)$
and for every $r>\max\{2, p'\}$ there is a constant
$C_{p, r}>0$ (actually the same as in \eqref{eq:24}) such that
\begin{align}
\label{eq:25}
\sup_{J \in \ZZ_+} \sup_{ I \in \mathfrak S_{J}(\RR_+) }
\norm[\big]{O^r_{I, J}(\calC_t f: t\in\RR_+)}_{L^p(\RR)}\le C_{p, r}\norm{f}_{L^p(\RR)}, \qquad f\in L^p(\RR).
\end{align}
For applications of \eqref{eq:24} and \eqref{eq:25} to the Wiener-Wintner theorem in ergodic theory, see \cite{LT} and 
\cite{OSTTW}.

It has been observed by M. Lacey \cite{lacey} (see \cite{RW} for the case $p=2$)
that \eqref{eq:25} remains true for  $r=2$ whenever $p\in[2, \infty)$. Furthermore, 
this can be extend to all $p>1$ when we restrict the $t$ parameter in $\calC_t$ to dyadic numbers $t\in {\mathbb D}$.
Our aim here is to show how these results follow as an immediate consequence of Theorem \ref{thm:p}. 

\begin{proposition} \label{prop:car}
Let $(\calC_t)_{t \in \RR_+}$ be  as in \eqref{eq:23}. 
Then for every $p\in[2, \infty)$,
 there exists a constant $C_p>0$ such that 
\begin{align}
\label{eq:27}
\sup_{J \in \ZZ_+} \sup_{ I \in \mathfrak S_{J}(\RR_+) }
\norm[\big]{O^2_{I, J}(\calC_t f: t\in\RR_+)}_{L^p(\RR)}
\le C_p 
\norm{ f }_{L^p(\RR)}, 
\qquad f \in L^p(\RR).
\end{align}
Furthermore for $(\calC_t)_{t\in {\mathbb D}}$, we have
\begin{align}
\label{eq:27+}
\sup_{J \in \ZZ_+} \sup_{ I \in \mathfrak S_{J}({\mathbb D}) }
\norm[\big]{O^2_{I, J}(\calC_t f: t\in{\mathbb D})}_{L^p(\RR)}
\le C_p 
\norm{ f }_{L^p(\RR)}, 
\qquad {\rm for \ all} \ p \in(1, \infty).
\end{align}
\end{proposition}

\begin{proof}
Observe the operators  $(\calC_t)_{t \in \RR_+}$ are projections in the sense of \eqref{eq:10}. 
Moreover when the sequence $(I_j)_{j\in\NN} \subset {\mathbb D}$ lies among the dyadic numbers, the bound
\begin{align*}
\sup_{J\in\ZZ_+}\sup_{I\in \mathfrak{S}_J({\mathbb D})} 
\norm[\Big]{ \Big( \sum_{j=0}^{J-1} 
\abs{ (\calC_{I_{j+1}}- \calC_{I_{j}}) f}^2 \Big)^{1/2} }_{L^p(\RR)}
\lesssim_{p} 
\norm{f}_{L^p(\RR)}, \qquad p \in(1, \infty),
\end{align*}
follows from the classical Littlewood--Paley inequality associated to dyadic intervales (no need to refer to the refinements of the theory from Section \ref{lp}). This
verifies \eqref{1.15} with $r=2$ and $p\in(1, \infty)$ in the dyadic case.
Furthermore by Rubio de Francia's square function  theorem for intervals (see Section \ref{lp}), one has for every $p\in[2, \infty)$ that
\begin{align*}
\sup_{J\in\ZZ_+}\sup_{I\in \mathfrak{S}_J(\RR_+)} 
\norm[\Big]{ \Big( \sum_{j=0}^{J-1} 
\abs{ (\calC_{I_{j+1}}- \calC_{I_{j}}) f}^2 \Big)^{1/2} }_{L^p(\RR)}
\lesssim_{p} 
\norm{f}_{L^p(\RR)}, \qquad f \in L^p(\RR),
\end{align*}
which verifies \eqref{1.15} with $r=2$ and $p\in[2, \infty)$. Using
\eqref{eq:26} with $r=2$ we also see that \eqref{1.2} is verified with
$r=2$ and $p\in (1, \infty)$. Thus invoking Theorem \ref{thm:p}
inequalities \eqref{eq:27} and \eqref{eq:27+} follow.
\end{proof}

Proposition \ref{prop:car} for $p=2$ was established 
by Rosenblatt and Wierdl \cite[inequality (4.12), p. 82]{RW}. 
In \cite{LT}, Lacey and Terwilleger established \eqref{eq:27+} for $p\in (1,\infty)$. Proposition \ref{prop:car} gives a simple proof of these results.

In view of inequality \eqref{eq:9} it is not difficult to see that the maximal estimates \eqref{eq:21} and the oscillation estimates \eqref{eq:27} for the Carleson operator are equivalent for all $p\in[2, \infty)$.

We also remark that the proof above also gives a proof of \eqref{eq:25} which does not appeal to the variational inequality \eqref{eq:24}. 
Indeed, Rubio de Francia's result in \cite[inequality (7.1), p. 10]{RdF} states that for every $p\in(1, 2)$ and $r>p'$ one has
\begin{align}
\label{eq:54}
\sup_{J\in\ZZ_+}\sup_{I\in \mathfrak{S}_J(\RR_+)} 
\norm[\Big]{ \Big( \sum_{j=0}^{J-1} 
\abs{ (\calC_{I_{j+1}}- \calC_{I_{j}}) f}^r \Big)^{1/r} }_{L^p(\RR)}
\lesssim_{p,r} 
\norm{f}_{L^p(\RR)}, \qquad f \in L^p(\RR).
\end{align}
Hence using \eqref{eq:54} and \eqref{eq:26} and invoking Theorem \ref{thm:p} we obtain the desired claim in \eqref{eq:25}.

A counterexample of Cowling and Tao \cite{CT} to Rubio de Francia's conjecture in \cite[Conjecture 7.2]{RdF}  shows that for all $p\in(1, 2)$, one has
\begin{align*}
\sup_{\|f\|_{L^p(\RR)}\le1}\sup_{I\in \mathfrak{S}_\infty(\RR_+)} 
\norm[\Big]{ \Big( \sum_{j=0}^{\infty} 
\abs{ (\calC_{I_{j+1}}- \calC_{I_{j}}) f}^{r} \Big)^{1/{r}} }_{L^p(\RR)}=\infty, \qquad \text{ where } \qquad r=\frac{p}{p-1}.
\end{align*}
Therefore \eqref{eq:25} for $r=p'$ with $p\in(1, 2)$ cannot hold. This shows that the range of $p$ and $r$ in \eqref{eq:25} and \eqref{eq:27} is sharp.

\section{Multi-parameter oscillation estimates}
\label{section:4}
In this section we establish Theorem \ref{thm:main}.  We begin with
proving an abstract multi-parameter
oscillation result, which may be of independent interest. Before we do
this we need more notation. For linear operator $T: L^0(X)\to L^0(X)$
we shall denote by $\abs{T}$ the sublinear maximal operator taken in
the lattice sense defined by
\[
\abs{T} f (x) = \sup_{\abs{g} \le \abs{f}} \abs{Tg (x)}, \qquad x \in X, \text{ and }  f\in L^p(X).
\]
For two linear operators $S,T: L^0(X)\to L^0(X)$ we have $\abs{ST}f \le \abs{S} \abs{T}f$ whenever $f\in L^0(X)$.
\begin{proposition} \label{prop:8}
Let $(X,\mathcal B(X), \mu)$ be a $\sigma$-finite measure space and let $\II\subseteq \RR$ be such that $\#\II\ge2$.
Let $k \in \NN_{\ge2}$ and  $p, r\in(1, \infty)$ be fixed. Let $(T_t)_{t\in\II^k}$ be a family of linear operators of the form
\begin{align*}
T_t :=  T_{t_1}^{1} \cdots T_{t_k}^{k},
\qquad t=(t_1,\ldots, t_k)\in\II^k,
\end{align*}
where
$\{ T_{t_i}^i : i \in [k], \, t_i \in \II \}$ is a family of commuting linear operators, which are bounded on $L^p(X)$.
If the set $\II$ is uncountable then we also assume that $\II\ni t\mapsto T_t^if$ is
continuous $\mu$-almost everywhere on $X$ for every $f\in L^0(X)$ and  $i\in[k]$. Further assume that for every $i\in[k]$, we have
\begin{align}
\label{eq:28}
\sup_{J \in \ZZ_+} \sup_{ I \in \mathfrak S_{J}(\II)}
\norm{ O^r_{I, J} ( T_{t}^{i} f : t \in \II) }_{L^p(X)} 
& \lesssim_{p, r} 
\norm{ f }_{L^p(X)}, 
\qquad f \in L^p(X), 
\end{align}
and
\begin{align}
\label{eq:29}
\norm[\Big]{ \Big( \sum_{j \in \ZZ} \big( \sup_{t \in \II} \abs{T_t^i} \abs{f_j} \big)^r \Big)^{1/r} }_{L^{p}(X)}
& \lesssim_{p, r}
\norm[\Big]{ \Big( \sum_{j \in \ZZ} \abs{f_j}^r  \Big)^{1/r} }_{L^{p}(X)}, \qquad (f_j)_{j \in \ZZ} \in L^{p}(X;\ell^r(\ZZ)).
\end{align}
Then we have the following multi-parameter $r$-oscillation estimate:
\begin{align*}
\sup_{J \in \ZZ_+} \sup_{ I \in \mathfrak S_{J}(\II^k) }
\norm{ O^r_{I, J} ( T_{t} f : t \in \II^k) }_{L^p(X)} 
\lesssim
\norm{ f }_{L^p(X)}, 
\qquad f \in L^p(X).
\end{align*}
\end{proposition}

\begin{proof}
For $i \in [k]$ and $n = (n_1, \ldots, n_{i-1}, n_{i+1}, \ldots, n_{k}) \in \II^{k-1}$, let us denote 
\[
T_{n}^{(i)}:  = T_{n_1}^{1} \cdots T_{n_{i-1}}^{i-1} T_{n_{i+1}}^{i+1} 
\cdots T_{n_k}^{k}.
\]
Using this definition, the bound \eqref{eq:29} and proceeding inductively we easily see that
\begin{align} \label{3.3.2}
\norm[\Big]{ \Big( \sum_{j \in \ZZ} \big( \sup_{n \in \II^{k-1}} \abs{T_{n}^{(i)}} \abs{f_j} \big)^r \Big)^{1/r} }_{L^{p}(X)}
& \lesssim_p
\norm[\Big]{ \Big( \sum_{j \in \ZZ} \abs{f_j}^r  \Big)^{1/r} }_{L^{p}(X)},
\end{align}
uniformly in $i \in [k]$ and $(f_j)_{j \in \ZZ} \in L^{p}(X;\ell^r(\ZZ))$.
Furthermore for $n \in \II^k$ and $I_j=(I_{j1},\ldots, I_{jk})\in\II^k$,  we have the identity
\begin{align}
\label{eq:30}
T_n f - T_{I_{j}} f 
=
\sum_{m=1}^k T_{n(m,n,I_{j})}^{(m)} (T_{n_{m}}^{m} - T_{I_{jm}}^{m}) f,
\end{align}
where 
$n(m,n,I_{j}) := (n_1, \ldots, n_{m-1}, I_{j(m+1)}, \ldots, I_{jk}) \in \II^{k-1}$.

We now fix $J\in\ZZ_+$ and a sequence $I \in \mathfrak S_{J}(\II^k)$.
Applying the identity \eqref{eq:30}, the triangle inequality, the bound \eqref{3.3.2} applied to 
$f_j^m = \sup_{\substack{ I_{jm} \le n_m < I_{(j+1)m} \\ n_m \in \II}}
\abs{T_{n_m}^{m} f - T_{I_{jm}}^{m} f}$ and \eqref{eq:28}, we obtain
\begin{align*} 
& \norm[\Big]{ \Big( \sum_{j =0}^{J-1}   \sup_{n \in \BB[I_{j}]\cap\II^k} \abs{T_{n}f - T_{I_{j}}f}^r \Big)^{1/r} }_{L^{p}(X)} \\
& \qquad \le 
\sum_{m=1}^k \norm[\Big]{ \Big( \sum_{j =0}^{J-1} 
	\big(  \sup_{n \in \BB[I_{j}]\cap\II^k}\abs{T_{n(m,n,I_{j})}^{(m)}} \abs{T_{n_m}^{m} f - T_{I_{jm}}^{m} f} \big)^r \Big)^{1/r} }_{L^{p}(X)} \\
& \qquad \le 
\sum_{m=1}^k \norm[\Big]{ \Big( \sum_{j =0}^{J-1}  
 \Big( \sup_{n \in \II^{k-1}}  \abs{T_{n}^{(m)}} 
\big( \sup_{\substack{ I_{jm} \le n_m < I_{(j+1)m} \\ n_m \in \II}}
\abs{T_{n_m}^{m} f - T_{I_{jm}}^{m} f} \big) \Big)^r \Big)^{1/r} }_{L^{p}(X)} \\
& \qquad \lesssim
\sum_{m=1}^k \norm[\Big]{ \Big( \sum_{j =0}^{J-1}  
\sup_{\substack{ I_{jm} \le n_m < I_{(j+1)m} \\ n_m \in \II}}
 \abs{T_{n_m}^{m} f - T_{I_{jm}}^{m} f}^r \Big)^{1/r} }_{L^{p}(X)} \ \lesssim \ \norm{f}_{L^{p}(X)}.
\end{align*}
This completes the proof of Proposition~\ref{prop:8}.
\end{proof}

We have a simple consequence of the above result.

\begin{corollary} \label{cor:2} Let $k \in \NN_{\ge2}$ and fix parameters
$n_1, \ldots, n_k \in \ZZ_+$, and  $p\in(1, \infty)$.  For every  $i\in[k]$ let
$\phi^{i}:\RR^{n_i}\to \CC$ be a Schwartz function, and define
$\phi_{t_i}^{i}(x):= t_{i}^{-n_i} \phi^{i} (t_i^{-1}x)$ for every
$t_i \in\RR_+$ and $x \in \RR^{n_i}$.
Set $N := n_1 + \ldots + n_k$ and for $t = (t_1, \ldots, t_k)\in \RR_+^k$ and $x = (x_1, \ldots, x_k) \in \RR^N:=\RR^{n_1}\times\ldots\times\RR^{n_k}$, consider the operator $T_{t}: L^p(\RR^{N})\to L^p(\RR^{N})$ defined by
\begin{align*} 
T_{t} f (x) 
:= \int_{\RR^{n_1}}\ldots\int_{\RR^{n_k}} \Big( \prod_{i=1}^k \phi_{t_i}^{i}(z_i) \Big) f(x-z) \, dz_1\ldots dz_k, \qquad z = (z_1, \ldots, z_k).
\end{align*}
Then  we have the following multi-parameter oscillation estimate
\begin{align} \label{id:3}
\sup_{J \in \ZZ_+} \sup_{ I \in \mathfrak S_{J}(\RR_+^k) }
\norm{ O^2_{I, J} ( T_{t} f : t \in \RR_+^k) }_{L^p(\RR^N)} 
\lesssim_p 
\norm{ f }_{L^p(\RR^N)}, 
\qquad f \in L^p(\RR^N).
\end{align}
\end{corollary}

\begin{proof}
For $i \in [k]$ and $z_i \in \RR^{n_i}$ we denote by $z_i^{(i)}=(z^{(i)}_1,\ldots, z^{(i)}_k)$ the point in $\RR^N=\RR^{n_1}\times\ldots\times\RR^{n_k}$ such that $z^{(i)}_j = \ind{\{j\}}(i) z_i\in \RR^{n_i}$ for any $j \in [k]$. We define the operators $T_{t_i}^i: L^p(\RR^{N})\to L^p(\RR^{N})$ by
\begin{align*} 
T_{t_i}^i f (x) :
= \int_{\RR^{n_i}} \phi_{t_i}^{i}(z_i) f(x-z_i^{(i)}) \, dz_i,
\qquad x = (x_1, \ldots, x_k) \in \RR^N, \quad t_i \in\RR_+.
\end{align*}
These operators commute and we have
$T_{t} = T_{t_1}^1 \circ\ldots\circ T_{t_k}^k$.
Furthermore, these are partial convolution operators with Schwartz functions and so Theorem \ref{thm:1} (see Remark \ref{rem:partial})
implies that the oscillation estimate \eqref{eq:28} holds for the family $(T^i_t)_{t \in \RR_+}$, for each $i \in [k]$. 
Finally, the Fefferman--Stein
vector-valued maximal inequality shows that \eqref{eq:29} holds and so
Proposition~\ref{prop:8} gives us the desired conclusion \eqref{id:3}. This completes the proof of
Corollary~\ref{cor:2}.
\end{proof}

We close this section by establishing the main ergodic result of this survey.

\begin{proof}[Proof of Theorem \ref{thm:main}]
We will invoke Proposition \ref{prop:8} with $k=d$ and $r=2$. As in \eqref{eq:50} note that
\begin{align*}
A_{M; X, \calT}^{P_1 ({\mathrm m}_1),\ldots, P_d ({\mathrm m}_d)}f=A_{M_1,\ldots, M_d; X, T_1,\ldots, T_d}^{P_1 ({\mathrm m}_1),\ldots, P_d ({\mathrm m}_d)}f=A_{M_1; X, T_1}^{P_1({\mathrm m}_1)}\circ\ldots\circ A_{M_d; X, T_d}^{P_d({\mathrm m}_d)}f,
\end{align*}
where the averages
$A_{M_1; X, T_1}^{P_1({\mathrm m}_1)},\ldots, A_{M_d; X, T_d}^{P_d({\mathrm m}_d)}$ commute. Thus it
remains to verify \eqref{eq:28} and \eqref{eq:29}. We fix $j\in[d]$. For \eqref{eq:28} we refer to \cite[Theorem 1.4]{MSS}, which ensures that for every $p\in(1,\infty)$ one has
\begin{align*}
\sup_{J \in \ZZ_+} \sup_{ I \in \mathfrak S_{J}(\ZZ_+)}
\norm{ O^2_{I, J} ( A_{M_j; X, T_j}^{P_j({\mathrm m}_j)} : M_j \in \ZZ_+) }_{L^p(X)} 
& \lesssim_{p} 
\norm{ f }_{L^p(X)}, 
\qquad f \in L^p(X).
\end{align*}
For \eqref{eq:29} we refer to \cite[Theorem C]{MST1}, which guarantees that for every $p\in(1,\infty)$ one has
\begin{align*}
\norm[\Big]{ \Big( \sum_{\iota \in \ZZ} \big( \sup_{M_j \in \ZZ_+} \abs{A_{M_j; X, T_j}^{P_j({\mathrm m}_j)}}|f_{\iota}| \big)^2 \Big)^{1/2} }_{L^{p}(X)}
& \lesssim_{p}
\norm[\Big]{ \Big( \sum_{\iota \in \ZZ} \abs{f_\iota}^2  \Big)^{1/2} }_{L^{p}(X)}, \qquad (f_j)_{j \in \ZZ} \in L^{p}(X;\ell^2(\ZZ)).
\end{align*}
This completes the proof of the  multi-parameter oscillation inequality \eqref{eq:52} in Theorem \ref{thm:main}. 
\end{proof}


\begin{thebibliography}{999}


\bibitem{Bel} \textsc{A. Bellow}.
\newblock {Measure Theory Oberwolfach 1981. Proceedings of the Conference held at Oberwolfach, June 21--27, 1981.}
\newblock {Lecture Notes in Mathematics {\bf 945}, editors D. K\"olzow and D. Maharam-Stone.  Springer-Verlag Berlin Heidelberg (1982).}
\newblock {Section: Two problems submitted by A. Bellow, 429--431.} 

\bibitem{BORSS} \textsc{D. Beltran, R. Oberlin, L. Roncal, A. Seeger, B. Stovall}. 
\newblock {Variation bounds for spherical averages.}
\newblock {Math. Ann. {\bf 382} (2022), 459--512.}

\bibitem{BRS} \textsc{D. Beltran, J. Roos, A. Seeger}. 
\newblock {Multi-scale sparse domination.}
\newblock {Memoirs of the AMS}, to appear.

\bibitem{BL} \textsc{V. Bergelson, A. Leibman}.
\newblock {A nilpotent Roth theorem.}
\newblock {{Invent. Math.} {\bf 147} (2002), 429--470.}



\bibitem{BI}
\textsc{G. Birkhoff}.
\newblock{Proof of the ergodic theorem.}
\newblock{Proc. Natl. Acad. Sci. USA {\bf 17} (1931), no. 12,  656--660.}


\bibitem{B1} \textsc{J. Bourgain.}
\newblock {On the maximal ergodic theorem for certain subsets of the integers.}
\newblock {{Israel J. Math.} {\bf 61} (1988), 39--72.}


\bibitem{B2} \textsc{J. Bourgain.}
\newblock {On the pointwise ergodic theorem on $L^p$ for arithmetic sets.}
\newblock {{Israel J. Math.} {\bf 61} (1988), 73--84.}


\bibitem{B3} \textsc{J. Bourgain.}
\newblock{Pointwise ergodic theorems for arithmetic sets. With an appendix by the author, H. Furstenberg, Y. Katznelson, and D.S. Ornstein.}
\newblock {{Inst. Hautes Etudes Sci. Publ. Math.} {\bf 69} (1989), 5--45.}


\bibitem{BMSWer}
\textsc{J. Bourgain, M. Mirek, E. M. Stein, J. Wright.}
\newblock {On a multi-parameter variant of the Bellow--Furstenberg problem.}
\newblock {Preprint, (2022).}



\bibitem{BM} \textsc{Z. Buczolich, R.D. Mauldin.}
\newblock {Divergent square averages.}
\newblock {{Ann. Math.} {\bf 171} (2010), no. 3,  1479--1530.}


\bibitem{Bu1} \textsc{D.L. Burkholder.}
\newblock {Martingale transforms.}
\newblock { Ann. Math. Statist. \textbf{37} (1966), 1494--1504.}

\bibitem{Bu2} \textsc{D.L. Burkholder.}
\newblock {Explorations in martingale theory and its applications. }
\newblock { {\'E}cole d'{\'E}t{\'e} de Probabilit{\'e}s de Saint-Flour XIX-1989, 1--66, Lecture Notes in Math., 1464, Springer, Berlin, 1991.}


\bibitem{Cald}
\textsc{A. Calder\'{o}n.}
\newblock{ Ergodic theory and translation invariant operators.}
\newblock{ Proc. Natl. Acad. Sci. USA {\bf 59} (1968),  349--353.}


\bibitem{CJRW} \textsc{J.T. Campbell, R.L. Jones, K. Reinhold, M. Wierdl.}
\newblock{Oscillation and variation for the Hilbert transform.}
\newblock{{Duke Math. J.} {\bf 105} (2000), no. 1, 59--83.}



\bibitem{Carleson} \textsc{L. Carleson.}
\newblock {On convergence and growth of partial sums of Fourier series.}
\newblock {Acta Math. {\bf 116} (1966), 135--157.}


\bibitem{CT} \textsc{M. Cowling, T. Tao.}
\newblock {Some light on Littlewood–Paley theory.}
\newblock {Math. Ann. {\bf 321} (2001), 885--888.}

\bibitem{DDU} \textsc{F. Di Plinio, Y. Do, G. Uraltsev.}
\newblock{Positive sparse domination of variational Carleson operators.}
\newblock{Ann. Sc. Norm. Super. Pisa Cl. Sci. (5) {\bf 18} (2018), no. 4, 1443--1458.}


\bibitem{DL}  \textsc{Y. Do, M. Lacey}
\newblock{Weighted bounds for variational Fourier series.}
\newblock{Studia Math. {\bf 211} (2012), no. 2, 153--190.}


\bibitem{D}\textsc{N. Dunford.}
\newblock {An individual ergodic theorem for non-commutative transformations.}
\newblock {{Acta Sci. Math. Szeged  {\bf 14} (1951),  1--4.}}

\bibitem{DR}\textsc{J. Duoandikoetxea, J.L. Rubio de Francia.}
\newblock{Maximal and singular integral operators via Fourier transform estimates.}
\newblock{{Invent. Math.} {\bf 84} (1986),  541--561.}


\bibitem{EW} \textsc{M. Einsiedler, T. Ward.}
\newblock {Ergodic Theory with a view towards Number Theory.}
\newblock {Graduate Texts in Mathematics 259, Springer-Verlag London (2011).}



\bibitem{F1}\textsc{C. Fefferman.}
\newblock {Pointwise convergence of Fourier series. }
\newblock {Ann. of Math. {\bf 98} (1973), 551--571.}


\bibitem{F}\textsc{H. Furstenberg.}
\newblock {Problems Session, Conference on Ergodic Theory and Applications}
\newblock {University of New Hampshire, Durham, NH, June 1982.}


\bibitem{GMS} \textsc{L. Grafakos, J.M. Martell, F. Soria.}
\newblock{Weighted norm inequalities for maximally modulated singular integral operators.}
\newblock{Math. Ann. {\bf 331} (2005), 359--394.}


\bibitem{GTT}\textsc{L. Grafakos, T. Tao, E. Terwilleger.}
\newblock {$L^p$ bounds for a maximal dyadic sum operator. }
\newblock {Math. Zeit. {\bf 246} (2004), no. 1-2, 321--337.}

\bibitem{GRY} \textsc{S. Guo, J. Roos, P.-L. Yung.}
\newblock {Sharp variation-norm estimates for oscillatory integrals related to Carleson's theorem.} 
\newblock {Anal. PDE {\bf 13} (2020), 1457--1500.}

\bibitem{Hunt} \textsc{R. Hunt.}
\newblock {On the convergence of Fourier series, Orthogonal Expansions and their Continuous
Analogues}
\newblock{(Proc. Conf. Edwardsville, IL, 1967), pp. 235–255, Southern Illinois Univ. Press,
Carbondale, IL, 1968.}


\bibitem{HNVW}\textsc{T. Hyt\"onen, J. van Neerven, M. Veraar, L. Weis.}
\newblock {Analysis in Banach spaces. Vol. I. Martingales and Littlewood--Paley theory.}
\newblock { A Series of Modern Surveys in Mathematics [Results in Mathematics and Related Areas. 3rd Series. A Series of Modern Surveys in Mathematics], 63. Springer, Cham, 2016. xvi+614 pp.}


\bibitem{IMMS} \textsc{A.D. Ionescu, {\'A}. Magyar, M. Mirek, T.Z. Szarek.}
\newblock {Polynomial averages and pointwise ergodic theorems on nilpotent groups.}
\newblock {To appear in the Inventiones Mathematicae, 72 pages, arXiv:2112.03322.}



\bibitem{IMSW}
\textsc{A. Ionescu, \'A. Magyar, E.M. Stein,  S. Wainger.} 
\newblock{Discrete Radon transforms and applications to ergodic theory.} 
\newblock{Acta Math{.} {\bf 198} (2007), 231--298.}

\bibitem{IMW} \textsc{A. Ionescu, \'A. Magyar, S. Wainger.}
\newblock{Averages along polynomial sequences in discrete nilpotent
Lie groups: Singular Radon transforms. Advances in analysis: the
legacy of Elias M. Stein},
\newblock{Princeton Math. Ser., {\bf 50},
Princeton Univ. Press, Princeton, NJ, (2014), 146--188.}



\bibitem{IW} \textsc{A.D. Ionescu, S. Wainger.}
\newblock {$L^p$ boundedness of discrete singular Radon transforms.}
\newblock {{J. Amer. Math. Soc. {\bf 19} (2005), no. 2, 357--383.}}



\bibitem{jkrw} \textsc{R.L. Jones, R. Kaufman, J.M. Rosenblatt, M. Wierdl.}
\newblock {Oscillation in ergodic theory.}
\newblock {{Ergodic Theory
Dynam. Systems} \textbf{18} (1998), no. 4, 889--935.}


\bibitem{JR1} \textsc{R.L. Jones, K. Reinhold.} Oscillation and
  variation inequalities for convolution powers. {Ergodic Theory
Dynam. Systems}  {\bf 21} (2001), no. 6, 1809--1829.


\bibitem{JRW} \textsc{R.L. Jones,  J.M. Rosenblatt, M. Wierdl.}
\newblock {Oscillation inequalities for rectangles.}
\newblock {{Proc. Amer. Math. Soc.} \textbf{129} (2001), no. 5, 1349--1358.}

\bibitem{JRW1} \textsc{R.L. Jones,  J.M. Rosenblatt, M. Wierdl.}
\newblock {Oscillation in ergodic theory: higher dimensional results.}
\newblock {{Israel J. Math.} \textbf{135} (2003),  1--27.}


\bibitem{JSW} \textsc{R.L. Jones, A. Seeger, J. Wright.}
\newblock {Strong variational and jump inequalities in harmonic analysis.}
\newblock {{Trans. Amer. Math. Soc.} \textbf{360} (2008),  no. 12, 6711--6742.}

\bibitem{JG} \textsc{R.L. Jones, G. Wang.}
\newblock {Variation inequalities for the Fej{\'e}r and Poisson kernels.}
\newblock {{Trans. Amer. Math. Soc.} \textbf{356} (2004),  no. 11, 4493--4518.}

\bibitem{KT} \textsc{C.E. Kenig, P.A. Tomas.}
\newblock{Maximal operators defined by Fourier multipliers.}
\newblock{{ Studia Math.} \textbf{68} (1980),  79–83.} 

\bibitem{Kol1}  \textsc{A.N. Kolmogorov.} 
\newblock {Une s{\'e}rie de Fourier--Lebesgue divergente presque partout.}
\newblock{Fund. Math. {\bf 4} (1923), 324--328.}

\bibitem{Kol2}  \textsc{A.N. Kolmogorov.} 
\newblock {Une s{\'e}rie de Fourier-Lebesgue divergente partout.}
\newblock{C.R. Acad. Sci. Paris, {\bf 183} (1926), 1327--1329.}


\bibitem{Kr}
\textsc{B. Krause.}
\newblock {Polynomial Ergodic Averages Converge Rapidly: Variations on a Theorem of {B}ourgain.}
\newblock {To appear in the {Israel J. Math.}, arXiv:1402.1803.}



\bibitem{KMT}
\textsc{B. Krause, M. Mirek, T. Tao.}
\newblock {Pointwise ergodic theorems for non-conventional bilinear
polynomial averages.}
\newblock{Ann. of Math. (2) {\bf 195} (2022), 997--1109.}



\bibitem{La1} \textsc{M. Lacey.}
\newblock {Carleson's theorem: proof, complements, variations.}
\newblock {Publ. Mat. {\bf 48} (2004),
no. 2, 251--307.}

\bibitem{La2} \textsc{M. Lacey.}
\newblock {Sparse bounds for spherical maximal functions.}
\newblock {J. Anal. Math. {\bf 139} (2019),
	613--635.}


\bibitem{LT} \textsc{M. Lacey, E. Terwilleger.}
\newblock{A Wiener–Wintner theorem for the Hilbert transform.}
\newblock { {Ark. Mat.} {\bf 46} (2008), 315–336.}

\bibitem{LT1} \textsc{M. Lacey, C. Thiele.}
\newblock {A proof of boundedness of the Carleson operator.}
\newblock {Math. Res. Lett. {\bf 7} (2000), no. 4, 361--370.}

\bibitem{lacey} \textsc{M. Lacey. Personal communication.}

\bibitem{LaV1}
\textsc{P. LaVictoire.}
\newblock { Universally $L^1$-Bad Arithmetic Sequences.}
\newblock { {J. Anal. Math.} {\bf 113} (2011), no. 1,  241--263.}


\bibitem{Lep}
\textsc{D. L{\'e}pingle.}
\newblock { La variation d'ordre $p$ des semi-martingales.}
\newblock { {Z. Wahrscheinlichkeitstheorie und Verw. Gebiete.} {\bf 36} (1976), no. 4, 295--316.}

\bibitem{MSW} \textsc{{\'A}. Magyar, E.M. Stein, S. Wainger.}
\newblock {Discrete analogues in harmonic analysis: spherical averages.}
\newblock {{Ann. Math. {\bf 155} (2002), 189--208.}}



\bibitem{Men} \textsc{D. Menshov.}
\newblock{Sur les s{\'e}ries de fonctions orthogonales.}  
\newblock{Fund. Math. {\bf 4} (1923), 82--105.}

\bibitem{M10} \textsc{M. Mirek.}
\newblock{$\ell^p\big(\ZZ^d\big)$-estimates for discrete Radon
transform: square function estimates.}  
\newblock{ Anal. PDE {\bf 11} (2018), no. 3, 583--608.}




\bibitem{MSS} \textsc{M. Mirek, W. S{\l}omian, T.Z. Szarek.}
\newblock{Some remarks on oscillation inequalities.}
\newblock{Preprint, (2021), 24 pages, arXiv:2110.01149.}



\bibitem{MST1} \textsc{M. Mirek, E.M. Stein, B. Trojan.}
\newblock{$\ell^p(\ZZ^d)$-estimates for discrete operators
of Radon type: Maximal functions and vector-valued estimates}.
\newblock{J. Funct. Anal. {\bf 277} (2019), 2471--2521.}


\bibitem{MST2} \textsc{M. Mirek, E.M. Stein, B. Trojan.}
{$\ell^p(\ZZ^d)$-estimates for discrete operators of Radon
  type: Variational estimates}. Invent. Math.  {\bf 209} (2017), no. 3, 665--748.


\bibitem{MSZ1} \textsc{M. Mirek, E.M. Stein, P. Zorin-Kranich.}
\newblock{Jump inequalities via real interpolation. }
\newblock{ {Math. Ann.} \textbf{376} (2020), no. 1-2, 797--819.}


\bibitem{MSZ2} \textsc{M. Mirek, E.M. Stein, P. Zorin-Kranich.}
\newblock{A bootstrapping approach to jump inequalities and their applications. }
\newblock{{ Anal. PDE} \textbf{13} (2020), no. 2, 527--558.}


\bibitem{MSZ3} \textsc{M. Mirek, E.M. Stein, P. Zorin-Kranich.}
\newblock{Jump inequalities for translation-invariant operators of Radon type on $\ZZ^d$. }
\newblock{{ Adv. Math.} \textbf{365} (2020), art. 107065, 57 pp. }



\bibitem{MT}
\textsc{M. Mirek, B. Trojan.}
\newblock{Discrete maximal functions in higher dimensions and applications to ergodic theory.}
\newblock{Amer{.} J{.} Math{.} {\bf 138} (2016), 1495--1532.}



\bibitem{vN} \textsc{J. von Neumann.}
\newblock{Proof of the quasi-ergodic hypothesis.}
\newblock{Proc. Natl. Acad. Sci. USA {\bf 18} (1932),  70--82.}



\bibitem{Nevo} \textsc{A. Nevo.}  \newblock {Pointwise Ergodic
Theorems for Actions of Groups.}  \newblock {Handbook of Dynamical
Systems, Volume 1B, A. Katok, B. Hasselblatt (eds.), Chapter 13,
Elsevier Science (2005).}


\bibitem{OSTTW} \textsc{R. Oberlin, A. Seeger, T. Tao, C. Thiele, J. Wright.}
\newblock {A variation norm Carleson theorem.}
\newblock {J. Eur. Math. Soc. (JEMS) {\bf 14} (2012), no. 2, 421--464.}


\bibitem{Peter} \textsc{K. Petersen.}
\newblock {Ergodic Theory.}
\newblock {Cambridge Studies in Advanced Mathematics 2, Cambridge University Press, (1989).}


\bibitem{Pierce} \textsc{L.B. Pierce.}
\newblock {On superorthogonality.}
\newblock {{J.  Geom. Anal. }
\textbf{31} (2021),  7096--7183.}


\bibitem{PX} \textsc{G. Pisier, Q.H. Xu.}
\newblock {The strong $p$-variation of martingales and orthogonal series.}
\newblock {{Probab. Theory Related Fields} \textbf{77} (1988), no. 4, 497--514.}


\bibitem{PT} \textsc{M. Pramanik, E. Terwilleger.}
\newblock {A weak $L^2$ estimate for a maximal dyadic sum operator on $\mathbb{R}^n$.}
\newblock {Illinois J. Math. {\bf 47} (2003), no. 3, 775--813.}


\bibitem{Rad} \textsc{H. Rademacher.}
\newblock {Einige S{\"a}tze {\"u}ber Reihen von allgemeinen Orthogonalfunktionen.}
\newblock {{Math. Ann.} {\bf 87} (1922),  112--138.}


\bibitem{Riesz} \textsc{F. Riesz.}
\newblock {Some mean ergodic theorems.}
\newblock {{J. London Math. Soc.} {\bf 13} (1938),  274--278.}




\bibitem{RW} \textsc{J. Rosenblatt, M. Wierdl.}
\newblock {Pointwise ergodic theorems via harmonic analysis. In Proc. Conference on Ergodic Theory (Alexandria, Egypt, 1993).}
\newblock {London Mathematical Society Lecture Notes,  {\bf 205} (1995), 3--151.}

\bibitem{RdF} \textsc{J.L. Rubio de Francia.}
\newblock {A Littlewood--Paley inequality for arbitrary intervals.}
\newblock {Rev. Mat. Iberoamericana {\bf 1} (1985), 1--14.}





\bibitem{bigs} \textsc{E.M. Stein.}
\newblock {{H}armonic {A}nalysis: {R}eal-{V}ariable {M}ethods, {O}rthogonality,
		and {O}scillatory {I}ntegrals.}
\newblock {{Princeton University Press, (1993).}}


\bibitem{Sem1} \textsc{E. Szemer\'{e}di.}
\newblock {On sets of integers containing no $k$ elements in arithmetic progression.}
\newblock {{Acta Arith.} \textbf{27} (1975), 199--245.}



\bibitem{TaoIW} \textsc{T. Tao.}
\newblock {The Ionescu--Wainger multiplier theorem and the adeles.}
\newblock {Mathematika \textbf{67} (2021), 647--677.}



\bibitem{U} \textsc{G. Uraltsev.}
\newblock {Variational Carleson embeddings into the upper $3$-space.}
\newblock {Preprint, (2016), arXiv:1610.07657.}


\bibitem{zk}
\textsc{P. Zorin--Kranich.}
\newblock {Variation estimates for averages along primes and polynomials.}
\newblock { {J. Funct. Anal.} {\bf 268} (2015), no. 1,  210--238.}


\bibitem{Z}\textsc{A. Zygmund.}
\newblock {An individual ergodic theorem for non-commutative transformations.}
\newblock {{Acta Sci. Math. Szeged  {\bf 14} (1951),  103--110.}}



\end{thebibliography}
\end{document}